\documentclass[preprint]{elsarticle}

\usepackage{amssymb,amsmath,amsthm, color}
\usepackage{inputenc}
\usepackage[margin=3cm]{geometry}
\usepackage{url}

\newtheorem{theorem}{Theorem}[section]
\newtheorem{lemma}[theorem]{Lemma}
\newtheorem{define}[theorem]{Definition}
\newtheorem{cor}[theorem]{Corollary}
\newtheorem{prop}[theorem]{Proposition}
\newtheorem{remark}[theorem]{Remark}

\newcommand{\G}{\mathbb G}
\newcommand{\n}{\mathcal N}
\newcommand{\I}{\mathcal I}
\newcommand{\F}{\mathbb F}

\newcommand{\N}{\mathbb N}
\newcommand{\Z}{\mathbb Z}
\newcommand{\C}{\mathcal C}

\newcommand{\GL}{\mathrm{GL}}
\newcommand{\PGL}{\mathrm{PGL}}

\newcommand{\ord}{\mathrm{ord}}

\newcommand{\doublespace}

\begin{document}

\begin{frontmatter}

\title{Invariant theory of a special group action on irreducible polynomials over finite fields}

\author[UFMG]{Lucas Reis}
\ead{lucasreismat@gmail.com}
\fntext[UFMG]{Permanent address: Departamento de Matem\'{a}tica, Universidade Federal de Minas Gerais, Belo Horizonte, MG, 30123-970, Brazil.}
\address{School of Mathematics and Statistics, Carleton University, 1125 Colonel By Drive, Ottawa ON (Canada), K1S 5B6}
\journal{Elsevier}
\begin{abstract}
In the past few years, an action of $\PGL_2(\F_q)$ on the set of irreducible polynomials in $\F_q[x]$ has been introduced and many questions have been discussed, such as the characterization and number of invariant elements. In this paper, we analyze some recent works on this action and provide full generalizations of them, yielding final theoretical results on the characterization and number of invariant elements.  
\end{abstract}

\end{frontmatter}





\section{Introduction}
Let $\F_q$ be the finite field with $q$ elements, where $q$ is a power of a prime $p$. Recall that, for a polynomial $f=\sum_{i=0}^{n}a_ix^i\in \F_q[x]$ of degree $n$, its \emph{reciprocal} is defined as $f^{*}=x^nf(1/x)$. A polynomial $f$ is \emph{self-reciprocal} if it coincides with its reciprocal, i.e., $f=f^{*}=x^nf(1/x)$. For instance, $f=x+1$ is self-reciprocal. The irreducible polynomials that are self-reciprocal have been extensively studied in many aspects, such as their number, construction and characterization. For more details, see~\cite{MG90}. In this context, the following transformations on univariate polynomials over $\F_q$ can be viewed as generalizations of the reciprocal of a polynomial.

\begin{define}\label{def:mobius-action}
For $A\in \GL_2(\F_q)$ with $A=\left(\begin{matrix}
a&b\\
c&d
\end{matrix}\right)$ and $f\in \F_q[x]$ a polynomial of degree $k$, set
$$A\circ f=(bx+d)^kf\left(\frac{ax+c}{bx+d}\right).$$
Additionally, if $[A]$ denotes the class of $A$ in the group $\PGL_2(\F_q)$, for $f\in \F_q[x]$ a nonzero polynomial, set
$$[A]\circ f=c_{f, A}\cdot (A\circ f),$$
where $c_{f, A}\in \F_q^*$ is the element of $\F_q$ such that $c_{f, A}\cdot (A\circ f)$ is monic.
\end{define}
\noindent For $E=\left(\begin{matrix}0&1\\
1&0
\end{matrix}\right)$, $E\circ f$ is the reciprocal of the polynomial $f$ and $[E]\circ f$ is the \emph{monic reciprocal} of $f$. It is straightforward to check that, if $B=\lambda A$ with $\lambda\in \F_q^*$ (i.e., $[A]=[B]$), then $A\circ f$ and $B\circ f$ are equal (up to the scalar $c=\lambda^k\ne 0$) and so $[A]\circ f$ is well defined. In~\cite{ST12}, the authors show interesting properties of the compositions $A\circ f$ and $[A]\circ f$ in the case when $f$ is an irreducible polynomial of degree at least two. For each positive integer $k$, let $\I_k$ be the set of monic irreducible polynomials of degree $k$ over $\F_q$. From Lemma 2.2 of~\cite{ST12}, we can derive the following properties.

\begin{lemma}\label{lem:aux-mobius-action}
For any $A, B\in \GL_2(\F_q)$ and $f\in \I_k$ with $k\ge 2$, the following hold.
\begin{enumerate}[(i)]
\item $[A]\circ f$ is in $\I_k$,
\item $[A]\circ ([B]\circ f)=[AB]\circ f$,
\item if $[I]$ is the identity of $\PGL_2(\F_q)$, $[I]\circ f=f$.
\end{enumerate}
\end{lemma}

In particular, the group $\PGL_2(\F_q)$ \emph{acts} on the sets $\I_k$ with $k\ge 2$, via the compositions $[A]\circ f$. It is then natural to ask about the fixed points. 
\begin{define}
\begin{enumerate}
\item For $f\in \I_k$ with $k\ge 2$ and $[A]\in \PGL_2(\F_q)$, $f$ is $[A]$-invariant if $[A]\circ f=f$.
\item For $[A]\in \PGL_2(\F_q)$ and a positive integer $n\ge 2$, $\C_{A}(n):=\{f\in \I_n\,|\, [A]\circ f=f\}$ is the set of $[A]$-invariants of degree $n$.
\item For $[A]\in \PGL_2(\F_q)$ and a positive integer $n\ge 2$, $\n_{A}(n):=|\C_{A}(n)|$ is the number of $[A]$-invariants of degree $n$.
\end{enumerate}
\end{define}
In~\cite{ST12}, the authors discuss many questions such on the characterization and number of $[A]$-invariants. In particular, they obtain the following criteria for the fixed elements.

\begin{theorem}[see Theorem 4.2 of~\cite{ST12}]\label{thm:ST12-4.2} Let $A=\left(\begin{matrix}
a&b\\
c&d
\end{matrix}\right)$ be an element of $\GL_2(\F_q)$. For each nonnegative integer $r$, set $F_{A, r}(x)=bx^{q^r+1}-ax^{q^r}+dx-c$. For $f\in \I_k$ with $k\ge 2$, the following are equivalent.
\begin{enumerate}[(i)]
\item $[A]\circ f=f$, i.e., $f$ is $[A]$-invariant,
\item $f$ divides $F_{A, r}$ for some $r\ge 0$.
\end{enumerate}
\end{theorem}
Using this criterion, they show that there are infinitely many $[A]$-invariants in $\F_q[x]$. In addition, looking at the possible degrees of $[A]$-invariants, they obtain the following result.

\begin{theorem}[see Theorem 3.3 of~\cite{ST12}]\label{thm:ST12-3.3}
Let $H$ be a subgroup of $\PGL_2(\F_q)$ of order $D\ge 2$. Assume that $f\in \I_n$ is invariant by $H$, i.e., $[A]\circ f=f$ for every $[A]\in H$. Then $n=2$ or $n$ is divisible by $D$.
\end{theorem}

In particular, any $[A]$-invariant is either quadratic or has degree divisible by the order $D$ of $[A]$. They also obtain the asymptotic growth of the number of $[A]$-invariants of a fixed degree.

\begin{theorem}[see Theorem 5.3 of~\cite{ST12}]\label{thm:ST12-5.3} If $[A]\in \PGL_2(\F_q)$ has order $D\ge 2$, then  
$$\n_{A}(Dm)\approx \frac{\varphi(D)}{Dm}q^{m},$$
where $\varphi$ is the Euler Phi Function. In other words, $\lim\limits_{m\to \infty}\frac{\n_A(Dm)\cdot Dm}{q^m\varphi(D)}=1.$
\end{theorem}
After obtaining proving this identity, they claim that all the machinery used to obtain such result is sufficient to provide the exact value of $\n_{A}(Dm)$ but their proof is too lengthy and technical. Finally, they show that the set of polynomials that are fixed by the whole group $\PGL_2(\F_q)$ is quite trivial.

\begin{theorem}[see Proposition~4.8 of~\cite{ST12}]
If $f\in \F_q[x]$ is $[A]$-invariant for any $[A]\in \PGL_2(\F_q)$, then $q=2$ and $f=x^2+x+1$. In particular, only quadratic polynomials can be fixed by the whole group $\PGL_2(\F_q)$. 
\end{theorem}

Earlier, Garefalakis~\cite{Gar11} considers a composition that is "conjugated" to the composition $A\circ f$ given in Definition~\ref{def:mobius-action}. Namely, for $A\in \GL_2(\F_q)$ with $A=\left(\begin{matrix}
a&b\\
c&d
\end{matrix}\right)$ and $f\in \F_q[x]$ a polynomial of degree $k$, he defines
$A\star f=(cx+d)^kf\left(\frac{ax+b}{cx+d}\right)$. Observe that, for any $A\in \GL_2(\F_q)$ and $f\in \I_k$ with $k\ge 2$,
\begin{equation}\label{eq:conjugated-action}A\star f=A^T\circ f,\end{equation}
where $^T$ denotes the transpose and $A\circ f$ is as in Definition~\ref{def:mobius-action}. He obtains a characterization on the invariant elements in the case when $A$ is of the form $\left(\begin{matrix}
a&0\\
0&1\end{matrix}\right)$ or $\left(\begin{matrix}
1&b\\
0&1\end{matrix}\right)$, corresponding to the ``changes'' of variable $x\mapsto ax$ and $x\mapsto x+b$, respectively (see Theorems 1 and 3 of \cite{Gar11}). We point out that, using the relation of the compositions $A\star f$ and $A\circ f$ given in Eq.~\eqref{eq:conjugated-action}, his characterization on the invariants can be viewed as a particular case of Theorem~\ref{thm:ST12-4.2} (see Theorems 1 and 3 of~\cite{Gar11}). Nevertheless, using this characterization on invariant elements, he obtains exact enumeration formulas for the number of invariants of a given degree (see Theorems 2 and 4 of~\cite{Gar11}).

In~\cite{LR17}, following the work of \cite{Gar11} on the compositions $A\star f$, we obtain an alternative characterization of the fixed elements by some particular upper triangular matrices $A\in \GL_2(\F_q)$; the invariant polynomials appear as a composition of the form $F(g(x))$, where $g(x)$ is a polynomial of degree exactly the order of $A$ in $\GL_2(\F_q)$. In particular, we provide alternative proofs of some enumeration formulas that appear in~\cite{Gar11} and we show that, if $H\le \GL_2(\F_q)$ is a group of order $p^r$ with $r\ge 2$, there do not exist irreducible polynomials $f$ such that $A\star f=f$ for every element $A\in H$.

More recently, in~\cite{MP17}, the authors show that if $[A]\in \PGL_2(\F_q)$ is an \emph{involution} (i.e., $[A]^2=[I]$), the irreducible polynomials $f$ of degree $2m$ such that $[A]\circ f=f$ are the irreducible polynomials of the form $h^mF(\frac{g}{h})$, where $g/h$ is a rational function of degree $2$ and $F\in \F_q[x]$ has degree $m$. Motivated by the completion and generalization of these works, three fundamental questions arise:
\begin{itemize}
\item Given a subgroup $H$ of $\PGL_2(\F_q)$, which elements $f\in \mathcal I_n$ with $n\ge 2$ are fixed by $H$, i.e., $[A]\circ f=f$ for all $[A]\in H$?

\item Given $[A]\in \PGL_2(\F_q)$, how many irreducible polynomials of degree $n$ are $[A]$-invariant?

\item Is there a general connection between $[A]$-invariants and rational functions?
\end{itemize}

The aim of this paper is to unify and generalize all those recent works, giving complete answers to the questions above. Here we present a brief overview of our results. We show that, for a non cyclic subgroup $H$ of $\PGL_2(\F_q)$, the polynomials fixed by $H$ have degree at most two. We also obtain enumeration formulas for the number of $[A]$-invariants for any $[A]\in \PGL_2(\F_q)$ and show that, if $D$ is the order of $[A]$ in $\PGL_2(\F_q)$, the $[A]$-invariants of degree at least three arise from a special rational function of degree $D$ that depends only on $[A]$.

\section{Preliminaries}
In this section, we provide a background material that is frequently used along the paper.

\subsection{Auxiliary lemmas}\label{subsec:lemma-aux}
Here we provide some auxiliary results that were used in~\cite{ST12} to prove Theorem~\ref{thm:ST12-5.3}.  We use slightly different notations and, for more details, see Sections 4 and 5 of \cite{ST12}.
\begin{define}\label{def:F_Ar}
For $A = \left( \begin{array}{cc}a&b\\ c&d\end{array}\right) \in\GL_2(\F_q)$ and $r$ a non-negative integer, set
\[
F_{A,r}(x) := bx^{q^r+1}-ax^{q^r}+dx-c.
\]
\end{define}
From Theorem~4.5 of~\cite{ST12}, we have the following result.
\begin{lemma}\label{ST4.5}
Let $f$ be an irreducible polynomial of degree $Dm\ge 3$ such that $[A]\circ f=f$, where $D$ is the order of $[A]$. The following hold:
\begin{enumerate}[(i)]
\item There is a unique positive integer $\ell\le D-1$ such that $\gcd(\ell, D)=1$ and $f$ divides $F_{A, s}(x)$, where $s=\ell\cdot \frac{Dm}{D}=\ell\cdot m$.
\item For any $r\ge 1$, the irreducible factors of $F_{A, r}$ are of degree $Dr$, of degree $Dk$ with $k<r$, $r=km$ and $\gcd(m, D)=1$ and of degree at most $2$.
\end{enumerate}
\end{lemma}

\begin{lemma}[see \cite{ST12}, item (a) of Lemma 5.1]\label{power}
Let $r\ge 1$ and let $k$ be a divisor of $r$ such that $m:=r/k$ is relatively prime with $D$, the order of $[A]$. For $j$ such that $jm\equiv 1\pmod D$, the irreducible factors of $F_{A, r}(x)$ of degree $Dk$ are exactly the irreducible factors of $F_{A^j, k}(x)$ of degree $Dk$.
\end{lemma}

\subsection{Some properties of the compositions $A\circ f$}
Here we present some basic properties of the compositions $A\circ f$ that are further used. 
\begin{lemma}\label{lem:propertiesAf}
For $A\in \GL_2(\F_q)$ with $A = \left( \begin{array}{cc}a&b\\ c&d\end{array}\right)$ and $f, g\in \F_q[x]$ nonzero polynomials, the following hold.

\begin{enumerate}[(i)]
\item If $b=0$ or $f(a/b)\ne 0$, then $\deg(f)=\deg(A\circ f)$. 
\item if $f$ and $g$ have the same degree $n>1$ and $f\ne -g$, the polynomial $F=A\circ f+A\circ g$ divides $(bx+d)^n\cdot (A\circ (f+g))$.
\item $A\circ (f\cdot g)=(A\circ f)\cdot (A\circ g)$. In particular, if $f$ divides $g$, then $A\circ f$ divides $A\circ g$.
\end{enumerate}
\end{lemma}
\begin{proof}
\begin{enumerate}[(i)]
\item Suppose that $\deg(f)=n$ and write $f(x)=\sum_{i=0}^{n}a_ix^i$. We obtain $$A\circ f=\sum_{i=0}^{n}a_i(ax+c)^i(bx+d)^{n-i},$$ and so the coefficient of $x^n$ in $A\circ f$ equals $\sum_{i=0}^{n}a_ia^ib^{n-i}$, that turns out to be $a_na^n\ne 0$ if $b=0$ and $b^nf(a/b)$ if $b\ne 0$. 

\item Write $f=\sum_{i=0}^na_ix^i, g=\sum_{i=0}^nb_ix^i$ and let $0\le k\le n$ be the greatest integer such that $a_k\ne -b_k$. In particular, $a_i=-b_i$ for $k<i\le n$ and so
$$A\circ f+A\circ g=\sum_{i=0}^{k}(a_i+b_i)(ax+c)^i(bx+d)^{n-i}.$$
Additionally, $f+g$ is a polynomial of degree $k$. From definition, $$A\circ (f+g)=\sum_{i=0}^{k}(a_i+b_i)(ax+c)^i(bx+d)^{k-i}=(bx+d)^{n-k}\cdot (A\circ f+A\circ g),$$ and the result follows.
\item This item follows by direct calculations.
\end{enumerate}
\end{proof}

\subsection{Invariants and conjugations}
Here we establish some interesting relations between the polynomials that are invariant by two conjugated elements. We start with the following result.

\begin{lemma}\label{lem:conj-inv}
Let $A, B, P\in \GL_2(\F_q)$ such that $[B]=[P]\cdot [A]\cdot [P]^{-1}$. For any $k\ge 2$ and any $f\in \I_k$, we have that $[B]\circ f=f$ if and only if $[A]\circ g=g$, where $g=[P]^{-1}\circ f\in \I_k$. 
\end{lemma}
\begin{proof}
Observe that, from Lemma~\ref{lem:aux-mobius-action}, the following are equivalent:
\begin{itemize}
\item $[B]\circ f=f,$
\item $[P]\circ ([A]\circ ([P]^{-1}\circ f))=f,$
\item $[A]\circ ([P]^{-1}\circ f)=[P]^{-1}\circ f$.
\end{itemize}
\end{proof}
\begin{define}
For $A\in \GL_2(\F_q)$, $\C_A=\cup_{n\ge 2}\C_A(n)$ is the set of $[A]$-invariants.
\end{define}

Since the compositions $[A]\circ f$ preserve degree, we obtain the following result.

\begin{theorem}\label{thm:conjugates-main}
Let $A, B, P\in \GL_2(\F_q)$ such that $B=PAP^{-1}$. Let $\tau:\C_B\to \C_A$ be the map given by $\tau(f)=[P]^{-1}\circ f$. Then $\tau$ is a degree preserving one to one correspondence. Additionally, for any $n\ge 2$, the restriction of $\tau$ to the set $\C_A(n)$ rises to an one to one correspondence between $\C_A(n)$ and $\C_B(n)$. In particular, $\n_{A}(n)=\n_B(n)$.
\end{theorem}
\begin{proof}
From the previous lemma, $\tau$ is well defined and is an one to one correspondence. Additionally, since the compositions $[A]\circ f$ preserve degree, the restriction of $\tau$ to the set $\C_A(n)$ rises to an one to one correspondence between $\C_A(n)$ and $\C_B(n)$ and, since these sets are finite, they have the same cardinality, i.e., $\n_{A}(n)=\n_B(n)$.
\end{proof}

\begin{define}
For $G$ a subgroup of $\PGL_2(\F_q)$ and $f\in \I_k$ with $k\ge 2$, $f$ is \emph{$G$-invariant} if $[A]\circ f=f$ for any $[A]\in G$.
\end{define}

Of course, if $G$ is cyclic and generated by $[A]\in \PGL_2(\F_q)$, $f$ is $[A]$-invariant if and only if is $G$-invariant. From the previous theorem, the following corollary is straightforward.

\begin{cor}\label{cor:conjugates-main}
Let $G, H\in \PGL_2(\F_q)$ be groups with the property that there exists $P\in \GL_2(\F_q)$ such that $G=[P]\cdot H\cdot [P]^{-1}=\{[P]\cdot [A]\cdot [P]^{-1}\,|\, [A]\in H\}$. Then there is an one to one correspondence between the $G$-invariants and the $H$-invariants that is degree preserving.
\end{cor}

\subsection{On the conjugacy classes of $\PGL_2(\F_q)$}
The previous results show that the fixed points can be explored considering the \emph{conjugacy classes} of $\PGL_2(\F_q)$. This requires a good understanding on the algebraic structure of the elements in $\PGL_2(\F_q)$. Our aim is to arrive in simple matrices where the study of invariants can be treatable. We start with the following definition.

\begin{define}
For $A\in \GL_2(\F_q)$ such that $[A]\ne [I]$, $A$ is of {\bf type} $1$ (resp. $2$, $3$ or $4$) if its eigenvalues are distinct and in $\F_q$ (resp. equal and in $\F_q$, symmetric and in $\F_{q^2}\setminus \F_q$ or not symmetric and in $\F_{q^2}\setminus \F_q$). 
\end{define}

From definition, elements of type $3$ appear only in odd characteristic. Also, the types of $A$ and $\lambda\cdot A$ are the same for any $\lambda\in \F_q^*$. For this reason, we say that $[A]$ is of \emph{type $t$} if $A$ is of type $t$. As follows, any element $[A]\in \PGL_2(\F_q)$ is conjugated to a special element of type $t$, for some $1\le t\le 4$.

\begin{theorem}\label{thm:types}
Let $A\in \GL_2(\F_q)$ such that $[A]\ne [I]$ and let $a, b$ and $c$ be elements of $\F_q^*$. Let $A(a):=\left( \begin{array}{cc}a&0\\ 0&1\end{array}\right),\, \mathcal E:=\left( \begin{array}{cc}1&0\\ 1&1\end{array}\right),\, C(b):=\left( \begin{array}{cc}0&1\\ b&0\end{array}\right)$ and $D(c):=\left( \begin{array}{cc}0&1\\ c&1\end{array}\right)$. Then $[A]\in \PGL_2(\F_q)$ is conjugated to:
\begin{enumerate}[(i)]
\item $[A(a)]$ for some $a\in \F_q\setminus\{0, 1\}$ if and only if $A$ is of type $1$.
\item $[\mathcal E]$ if and only if $A$ is of type $2$.
\item $[C(b)]$ for some non square $b\in \F_q^*$ if and only if $A$ is of type $3$.
\item $[D(c)]$ for some $c\in \F_q$ such that $x^2-x-c\in \F_q[x]$ is irreducible if and only if $A$ is of type $4$.
\end{enumerate}
\end{theorem}

\begin{proof}
\begin{enumerate}
\item Suppose that $[A]$ is conjugated to $[A(a)]$ for some $a\ne 0,1$. In other words, there exist $P\in \GL_2(\F_q)$ and $\lambda\in \F_q^*$ such that $A=\lambda\cdot P\cdot A(a)\cdot P^{-1}$, hence $A$ and $\lambda A(a)$ are conjugated. In particular, these elements have the same eigenvalues, which are $\lambda, \lambda\cdot a \in \F_q$. Since $a\ne 1$, these eigenvalues are distinct and then $A$ is of type 1. Conversely, if $A$ is of type 1, it has distinct eigenvalues in $\F_q^*$ and then it is diagonalizable. Suppose that the eigenvalues are $\alpha$ and $\beta$. If we set $a=\alpha/\beta\ne 1$, we have that $\beta A$ is conjugated to $A(a)$, hence $[A]$ is conjugated to $[A(a)]$.

\item Suppose that $[A]$ is conjugated to $[\mathcal E]$. In particular, there exist $P\in \GL_2(\F_q)$ and $\lambda\in \F_q^*$ such that $A=\lambda\cdot P\cdot \mathcal E\cdot P^{-1}$, hence the characteristic polynomial of $A$ is the same of $\lambda \mathcal E$, which is $(x-\lambda)^2$. Therefore, $A$ has equal eigenvalues in $\F_q$, i.e., $A$ is of type $2$. Conversely, if $A$ is of type $2$, then it has equal eigenvalues, say $\lambda\in \F_q^*$. In particular, there exists $P\in \GL_2(\F_q)$ such that
$A=P\cdot \left( \begin{array}{cc}\lambda &0\\ 1&\lambda\end{array}\right) P^{-1}$, hence $\lambda^{-1}A=P\cdot Q \left( \begin{array}{cc}1&0\\ 1&1\end{array}\right)Q^{-1}P^{-1},$
where $Q=\left(\begin{matrix}\lambda&0\\0&1\end{matrix}\right)$. This shows that $[A]$ is conjugated to $[\mathcal E]$.

\item As in the previous items, it is straightforward to check that if $[A]$ is conjugated to $[C(b)]$, then the characteristic polynomial of $A$ is $x^2-b\lambda^2$ for some $\lambda \in \F_q^*$. Since $b$ is a non square, $A$ has symmetric eigenvalues in $\F_{q^2}\setminus \F_q$, hence is of type 3. Conversely, if $A$ is of type 3, then its characteristic polynomial is of the form $x^2-b$, for some non square $b\in \F_q$. In particular, $A$ and $C(b)$ are diagonalizable in $\F_{q^2}$, with same eigenvalues. In other words, $A=SDS^{-1}$ and $C(b)=TDT^{-1}$. Therefore $A$ and $C(b)$ are conjugated by an element $P\in \GL_2(\F_{q^2})$, i.e., $A=P\cdot C(b)\cdot P^{-1}$. Writing $P=\left(\begin{matrix}x&y\\z&w\end{matrix}\right)$, the equality $AP=P\cdot C(b)$ yields a linear system in four variables with coefficients in $\F_q$ (since $A, C(b)\in \GL_2(\F_q)$). Since it has a solution, the elements $x, y, z$ and $w$ lie in $\F_q$ and so $P\in \GL_2(\F_q)$. Hence $[A]=[P]\cdot [C(b)]\cdot [P]^{-1}$, where $[P]\in \PGL_2(\F_q)$.

\item If $[A]$ is conjugated to $[D(c)]$, the characteristic polynomial of $A$ equals $x^2-\lambda x-\lambda ^2 c$ for some nonzero $\lambda \in \F_q$. Since $x^2-x-c\in \F_q[x]$ is irreducible, so is $x^2-\lambda x-\lambda ^2 c$. In particular, $A$ has non symmetric eigenvalues in $\F_{q^2}\setminus \F_q$, hence is of type 4. Conversely suppose that $A$ is of type 4. Therefore, its characteristic polynomial is an irreducible polynomial $x^2-ax-b\in \F_q$ such that $a$ is nonzero. In particular, the characteristic polynomial of $a^{-1}A$ is $x^2-x-b/a^2$, an irreducible polynomial of degree two. It follows that $a^{-1}A$ and $D(-b/a^2)$ are diagonalizable in $\F_{q^2}$, with same eigenvalues. We proceed as in the previous item and conclude that $[A]$ is conjugated to $[D(-b/a^2)]$ by an element in $\PGL_2(\F_q)$. 
\end{enumerate}
\end{proof}

\begin{define}
An element $A\in \GL_2(\F_q)$ is in \textbf{reduced form} if it is equal to $A(a)$, $\mathcal E$, $C(b)$ or $D(c)$ for some suitable $a$, $b$ or $c$ in $\F_q$.
\end{define}

The study of invariants is naturally focused on the matrices in reduced form. In particular, we are mostly interested in the following "changes of variable": $x\mapsto ax, x\mapsto x+1, x\mapsto b/x$ and $x\mapsto \frac{c}{x+1}$.
We finish this section giving a complete study on the order of the elements in $\PGL_2(\F_q)$, according to their type.  

\begin{lemma}\label{lem:order}
Let $[A]$ be an element of type $t$ and let $D$ be its order in $\PGL_2(\F_q)$. The following hold:
\begin{enumerate}[(i)]
\item For $t=1$, $D>1$ is a divisor of $q-1$.
\item For $t=2$, $D=p$.
\item For $t=3$, $D=2$.
\item For $t=4$, $D$ divides $q+1$ and $D>2$. 
\end{enumerate}
\end{lemma}
\begin{proof}
Since any two conjugated elements in $\PGL_2(\F_q)$ have the same order, from Theorem~\ref{thm:types}, we can suppose that $A$ is in the reduced form. From this fact, items (i), (ii) and (iii) are straightforward. For item (iv), let $\alpha, \alpha^q$ be the eigenvalues of $A$. We that $[A]^D=[A^D]$ and then $[A]^D=[I]$ if and only if $A^D$ equals the identity element $I\in\GL_2(\F_q)$ times a constant. The latter holds if and only if $\alpha^D$ and $\alpha^{qD}$ are equal. Observe that $\alpha^D=\alpha^{qD}$ if and only if $\alpha^{(q-1)D}=1$. In particular, since $\alpha\in \F_{q^2}\setminus \F_q$, we have $D>1$ and $D$ divides $q+1$. If $D=2$, then $\alpha^{q}=-\alpha$, a contradiction since $A$ is not of type $3$.
\end{proof}

We also have the "converse" of this result.

\begin{prop}\label{prop:orders}
Let $q$ be a power prime and let $D>1$ be a positive integer.
The following hold.

\begin{enumerate}[(i)]
\item if $D$ divides $q-1$, there exists an element $[A]\in \PGL_2(\F_q)$ of type $1$ in reduced form such that its order equals $D$,

\item if $D>2$ divides $q+1$, there exists an element $[A]\in \PGL_2(\F_q)$ of type $4$ in reduced form such that its order equals $D$.
\end{enumerate}
\end{prop}

\begin{proof}
\begin{enumerate}[(i)]
\item Let $a\in \F_q^*$ be a primitive element and set $k=(q-1)/D$. Therefore, $a^k$ has multiplicative order $D$ and so $[A(a)]$ has order $D$ in $\PGL_2(\F_q)$.

\item Let $\theta\in \F_{q^2}$ be a primitive element and set $k=(q+1)/D$. Therefore, $\beta=\theta^{k}$ has multiplicative order $(q-1)\cdot D$. Let $x^2-ax+b\in \F_q[x]$ be the minimal polynomial of $\beta$. Since $D>2$, $\beta^{q}\ne \pm\beta$, hence $a\ne 0$. If we set $\beta_0=\beta/a$ and $c=b/a^2$, we see that $D(c)$ has eigenvalues $\beta_0, \beta_0^q$ and is diagonalizable over $\F_{q^2}$. Therefore, the order of $[D(c)]$ equals the least positive integer $d$ such that $\beta_0^d=\beta_0^{qd}$, i.e., $1=\beta_0^{(q-1)d}=\beta^{(q-1)d}$. Since $\beta$ has multiplicative order $(q-1)\cdot D$, it follows that $d=D$.
\end{enumerate}
\end{proof}

\section{On $G$-invariants: the noncyclic case}
As pointed out earlier, in~\cite{ST12}, the authors show that the irreducible polynomials fixed by the whole group $\PGL_2(\F_q)$ are necessarily quadratic, in contrast to the infinitely many irreducible polynomials that are fixed by a single element $[A]\in \PGL_2(\F_q)$. In this section, we show that this phenomena occurs in any noncyclic subgroup of $\PGL_2(\F_q)$. Namely, we have the following theorem.
\begin{theorem}\label{thm:chap6-main-3}
Let $G\le \PGL_2(\F_q)$ be a noncyclic group. If $f$ is a monic irreducible polynomial of degree $n$  such that $f$ is $[A]$-invariant for any $[A]\in G$, then $n=2$, i.e., $f$ is a quadratic polynomial.
\end{theorem}

The proof of this result requires a variety of methods. We start looking at the special case when $G$ is a $p$-subgroup of $\PGL_2(\F_q)$, where some results that are presented in~\cite{LR17} can be employed.

\begin{theorem}\label{thm:p-group}
Let $G\le \PGL_2(\F_q)$ be a group of order $p^r$, where $r\ge 2$. Then here do not exist irreducible polynomials that are $G$-invariant.
\end{theorem}

\begin{proof}
We observe that $\PGL_2(\F_q)$ has $q^3-q=q(q-1)(q+1)$ elements. Write $q=p^s$ with $s\ge 1$. If $s=1$, the result follows by emptiness. If $s>1$, we observe that any Sylow $p$-subgroup of $\PGL_2(\F_q)$ has order $p^s$ and $G$ is contained in one of them, say $G_0$. Moreover, we have the explicit description of one of these Sylow $p$-subgroups: $$\G_q:=\left\{\left(\begin{matrix}
1&0\\a&1
\end{matrix}\right)\,;\, a\in \F_q\right\}.$$
\noindent It is well known that Sylow subgroups are pairwise conjugated. In particular, $\G_q$ and $G_0$ are conjugated. This conjugation rises to a subgroup $H_0$ of $\G_q$ that is conjugated to $G$. Therefore, form Corollary~\ref{cor:conjugates-main}, we can suppose that $G$ is a subgroup of $\G_q$.

If we set $T(a)=\left(\begin{matrix}
1&0\\a&1
\end{matrix}\right)$ for $a\in \F_q$, we see that there is an one to one correspondence between the subgroups of $\G_q$ and the $\F_p$-vector subspaces of $\F_q$ (regarded as an $s$-dimensional $\F_p$-vector space). In particular, any subgroup $H$ of $\G_q$ is of the form $\{T(a)\,|\, a\in S\}$ for some $\F_p$-vector space $S\subseteq \F_q$: if $H$ has order $p^r$, $S$ has dimension $r$. Let $S_0\subseteq \F_q$ be the $\F_p$-vector space associated to $G$. Therefore, an irreducible polynomial $f$ is $G$-invariant if and only if $[T(a)]\circ f=f$ for any $a\in S_0$. Given $a\in S_0$, from definition, $[T(a)]\circ f=f$ if and only if $T(a)\circ f=\lambda \cdot f$ for some $\lambda \in \F_q^*$. In other words, $f(x+a)=\lambda\cdot f(x)$. A comparison on the leading coefficient on both sides of the last equality yields $\lambda=1$. Therefore, $f$ is $G$-invariant if and only if $f(x+a)=f(x)$ for any $a\in S_0$, where $S_0$ is an $\F_p$-vector space of dimension $r\ge 2$. According to Theorem 2.7 of~\cite{LR17}, there do not exist irreducible polynomials with this property.
\end{proof}

Recall that Theorem~\ref{thm:ST12-4.2} shows that the $[A]$-invariants correspond to the irreducible divisors of the polynomials $F_{A, r}$ that were previously introduced (see Definition~\ref{def:mobius-action}). Considering a noncyclic subgroup $G$ of $\PGL_2(\F_q)$, a polynomial $f$ that is $G$-invariant must satisfy $[A]\circ f=[B]\circ f=f$ for some $[A]$ and $[B]$ that are not power of each other. In particular, $f$ divides $F_{A, r}$ and $F_{B, s}$ for some $s, r\ge 0$. If $s=r$, the following lemma shows that the polynomial $f$ is very particular.

\begin{lemma}\label{lem:aux-divisor}
Let $r$ be a non-negative integer and $A_1, A_2\in \GL_2(\F_q)$ such that $[A_1]\ne [A_2]$ in $\PGL_2(\F_q)$. If $f\in\F_q[x]$ is an irreducible polynomial that divides $F_{A_1, r}$ and $F_{A_2, r}$, then $f$ has degree at most $2$.
\end{lemma}

\begin{proof}
Write $A_i=\left( \begin{array}{cc}a_i&b_i\\ c_i&d_i\end{array}\right)$, hence $F_{A_i, r}(x)=b_ix^{q^r+1}-a_ix^{q^r}+d_ix-c_i$ for $i=1,2$. Let $\alpha\in \overline{\F}_q$ be a root of $f(x)$. Suppose that $f$ has degree at least $3$. Therefore, $\alpha\not \in \F_q$ and so 
$$\alpha^{q^r}=\frac{d_1\alpha-c_1}{b_1\alpha-a_1}=\frac{d_2\alpha-c_2}{b_2\alpha-a_2}.$$
In particular, $(d_1\alpha-c_1)(b_2\alpha-a_2)=(d_2\alpha-c_2)(b_1\alpha-a_1)$ or, equivalently, $$(d_1b_2-d_2b_1)\alpha^2-(c_1b_2-c_2b_1+a_2d_1-a_1d_2)\alpha+(a_2c_1-a_1c_2)=0.$$
Since $f$ has degree at least $3$, $\alpha$ cannot satisfy a nontrivial polynomial identity of degree at most $2$ and so we conclude that 
$$d_1b_2-d_2b_1=a_2c_1-a_1c_2=c_1b_2-c_2b_1+a_2d_1-a_1d_2=0.$$
From $d_1b_2-d_2b_1=a_2c_1-a_1c_2=0$, we conclude that $(a_1, c_1)=\lambda (a_2, c_2)$ and $(b_2, d_2)=\lambda'(b_1, d_1)$ for some $\lambda, \lambda'\in \F_q$. From $c_1b_2-c_2b_1+a_2d_1-a_1d_2=0$, we obtain $$0=\lambda\lambda'b_1c_2-c_2b_1+a_2d_1-\lambda\lambda'a_2d_1=(\lambda\lambda'-1)(c_2b_1-a_2d_1).$$

\noindent If $\lambda\lambda'=1$, it follows that $A_1=\lambda A_2$ and so $[A_1]=[A_2]$. Otherwise, $c_2b_1-a_2d_1=0$ and so $(a_2, c_2)=\lambda'' (b_1, d_1)$ for some $\lambda''\in \F_q$. Therefore, $(a_1, c_1)=\lambda (a_2, c_2)=\lambda\lambda''(b_1, d_1)$, a contradiction since $A_1\in \GL_2(\F_q)$.
\end{proof}

Of course if $f$ is fixed by $[A]$ and $[B]$, $f$ certainly divides $F_{A, r}$ and $F_{B, s}$ for some $r, s\ge 0$ and the previous lemma shows that, for $r=s$, unless $[A]=[B]$, the polynomial $f$ is quadratic. However, we cannot be sure that we indeed have $r=s$. Nevertheless, the following argument can be used: if $f$ is $[B]$-invariant, $[B]\circ f=f$ and so $B\circ f$ equals $f$ times a constant. Since $f$ divides $F_{A, r}$, from Lemma~\ref{lem:aux-mobius-action}, $B\circ f$ divides $B\circ F_{A, r}$ and then $f$ divides $B\circ F_{A, r}$. We shall see that $B\circ F_{A, r}$ divides a polynomial of the form $sx^{q^r+1}-tx^{q^r}+ux-v$ for suitable $s, t, u$ and $v$ in $\F_q$ that are obtained from the elements $A$ and $B$. In particular, the elements $s, t, u$ and $v$ can be associated to a $2\times 2$ matrix with entries in $\F_q$. This idea is summarized as follows.
\begin{define}\label{def:sigma-product}
For $A=\left( \begin{array}{cc}a&b\\ c&d\end{array}\right)$ and $A_0=\left( \begin{array}{cc}a_0&b_0\\ c_0&d_0\end{array}\right)$ in $\GL_2(\F_q)$, the matrix
\[ \sigma(A, A_0)=\left( \begin{array}{cc}\sigma_1(A, A_0)&\sigma_2(A, A_0)\\ \sigma_3(A, A_0)&\sigma_4(A, A_0)\end{array}\right),\]
where
$$\begin{cases}
\sigma_1(A, A_0)=&-(ba_0c_0-aa_0d_0+dc_0b_0-cb_0d_0)\\
\sigma_2(A, A_0)=&ba_0^2-aa_0b_0+da_0b_0-cb_0^2\\
\sigma_3(A, A_0)=&-(bc_0^2-ac_0d_0+dd_0c_0-cd_0^2)\\
\sigma_4(A,A_0)=&ba_0c_0-ac_0b_0+dd_0a_0-cb_0d_0,
\end{cases}$$
is the \emph{$\sigma$-product} of $A$ and $A_0$.
\end{define}

\begin{remark}
The further computations of $\sigma$-products of elements in $\GL_2(\F_q)$ are performed using the \emph{Sage Software}. In general, the $\sigma$-products are taken over matrices whose entries are $0, 1$ or generic elements in a finite field. For this reason, the $\sigma$-products are further computed in the (commutative) ring of eight variables over the rationals: we define $$T(a, b, c, d, a_0, b_0, c_0, d_0)=(\sigma_1, \sigma_2, \sigma_3, \sigma_4),$$
where the functions $\sigma_i$ are given as in Definition~\ref{def:sigma-product}.\end{remark}

In the following proposition, we provide some basic properties of the $\sigma$-product in $\GL_2(\F_q)$ and its direct connection to the polynomials $F_{A, r}$.

\begin{prop}\label{prop:sigma}
For $A, A_0\in \GL_2(\F_q)$, set $\lambda=\det(A)$ and $\lambda_0=\det(A_0)$. The following hold:

\begin{enumerate}[(i)]
\item $\det(\sigma(A, A_0))=\lambda\lambda_0^2$ and, in particular, $\sigma(A, A_0)\in \GL_2(\F_q)$,

\item for any integer $r\ge 0$, $A_0\circ F_{A, r}$ divides $F_{\sigma(A, A_0), r}$,
\end{enumerate}
In particular, if $f\in \F_q[x]$ is an irreducible polynomial of degree at least $3$ such that $$[A]\circ f=[A_0]\circ f=f,$$ then $\sigma(A, A_0)=\varepsilon\lambda_0A$ for some $\varepsilon\in \{\pm 1\}$.
\end{prop}

\begin{proof}
Write $A=\left( \begin{array}{cc}a&b\\ c&d\end{array}\right)$ and $A_0=\left( \begin{array}{cc}a_0&b_0\\ c_0&d_0\end{array}\right)$. Item (i) follows by direct calculations. For item (ii), we observe that
$$F_{\sigma(A, A_0), r}=b(a_0x+c_0)^{q^r+1}-a(a_0x+c_0)^{q^r}(b_0x+d_0)+d(a_0x+c_0)(b_0x+d_0)^{q^r}-c(b_0x+d_0)^{q^r+1},$$
From definition, $F_{\sigma(A, A_0), r}$ equals $M\cdot (A_0\circ F_{A, r})$, where $M=1$ or $M=c_0x+d_0$, according to whether $c\ne 0$ or $c=0$, respectively. In particular, $A_0\circ F_{A, r}$ divides $F_{\sigma(A, A_0), r}$. 

We observe that, if $f$ has degree at least $3$ and satisfies $[A]\circ f=[A_0]\circ f=f$, then $f$ divides $F_{A, r}$ for some $r\ge 0$. Additionally, from the previous observations, $f$ also divides $A_0\circ F_{A, r}$. Since $A_0\circ F_{A, r}$ divides $F_{\sigma(A, A_0), r}$, it follows that $f$ also divides $F_{\sigma(A, A_0), r}$. From item (i), $\sigma(A, A_0)\in \GL_2(\F_q)$ and then, from Lemma~\ref{lem:aux-divisor}, we have that $[\sigma(A, A_0)]=[A]$, i.e., $\sigma(A, A_0)=\delta A$ for some $\delta\in \F_q^*$. Taking the determinant on both sides of the previous equality and using item (i), we conclude that $\delta=\pm \lambda_0$.
\end{proof}

\subsection{Proof of~Theorem~\ref{thm:chap6-main-3}}
For a group $G\le \PGL_2(\F_q)$, $G$ is of \emph{type} $t$, if $1\le t\le 4$ is the least positive integer such that $G$ contains an element $[A]\in \PGL_2(\F_q)$ of type $t$. We recommend the reader to recall the matrix-type notation in Theorem~\ref{thm:types}. We start with the following preliminary result. 

\begin{prop}\label{prop:order>4}
Suppose that $G\le \PGL_2(\F_q)$ is a non cyclic group of order at least $5$. Then no irreducible polynomial $f\in \F_q[x]$ of degree at least $3$ is $G$-invariant.
\end{prop}
\begin{proof}
Let $1\le t\le 4$ be the type of $G$ and let $[A]\in G$ be an element of type $t$ such that the order of $[A]$ is maximal through the elements of type $t$ in $G$. Since $G$ is not cyclic, there exists an element $[A']\in G\setminus \langle[A]\rangle$. Recall that, from Corollary~\ref{cor:conjugates-main}, there is a degree preserving one to one correspondence between the $G$-invariants and the $G'$-invariants, for any group $G'$ that is conjugated to $G$. Of course, conjugations in $\PGL_2(\F_q)$ preserves the order of elements. In particular, we can suppose that $A$ is in the reduced form. Write $A'=\left( \begin{array}{cc}a'&b'\\ c'&d'\end{array}\right)$.  Suppose that $f\in \F_q[x]$ is an irreducible polynomial of degree at least $3$ that is $G$-invariant. In particular, $[A]\circ f=[A']\circ f=f$ and then, from Proposition~\ref{prop:sigma}, it follows that
\begin{equation}\label{LD}\sigma(A', A)=\varepsilon \cdot \det (A)\cdot A',\end{equation}
for some $\varepsilon\in \{\pm 1\}$. We have four cases to consider:
\begin{enumerate}[(a)]
\item $t=1$; in this case, there exists $a\in \F_q\setminus\{0, 1\}$ such that $A=A(a)$, i.e., $A=\left( \begin{array}{cc}a&0\\ 0&1\end{array}\right)$. A simple calculation yields $$\sigma(A', A)=\left( \begin{array}{cc}aa'&a^2b'\\ c'&ad'\end{array}\right).$$ 
From Eq.~\eqref{LD}, we obtain $\sigma(A', A)=\varepsilon a A'$. We split into cases:
\begin{enumerate}[$\bullet$]
\item {\bf Case a-1.} $\varepsilon=1$. We get $c'(a-1)=b'(a^2-a)=0$. Since $a\ne 0, 1$, we get $c'=b'=0$, and then $[A']=[A(a_0)]$, where $a_0=a'/d'$. However, since $\F_q^*$ is cyclic, then so is the group generated by elements of the form $[A(s)], s\in \F_q^*$: if $\theta$ is any generator of $\F_q^*$, $[A(\theta)]$ generates the elements $[A(s)], s\in \F_q^*$. In particular, $\langle [A], [A']\rangle$ is a cyclic group generated by the element $[A(a_0')]\in G$, for some $a_0\in \F_q^*$. Since $\ord([A])$ is maximal through the elements of type 1 in $G$, we conclude that $\langle [A], [A']\rangle= \langle [A]\rangle$, a contradiction with $[A']\in G\setminus  \langle [A]\rangle$. 

\item {\bf Case a-2}. $\varepsilon=-1$. We have seen that $\varepsilon\ne 1$, hence $-1\ne 1$, that is, $q$ is odd. We obtain $$2aa'=2ad'=(a^2+a)b'=(1+a)c'=0.$$ Since $q$ is odd and $a\ne 0$, it follows that $a'=d'=0$, $b', c'\ne 0$ and so $1+a=a^2+a=0$, i.e., $a=-1$. Therefore, $[A]=[A(-1)]$ and $[A']=[C(b_0)]$ with $b_0=c'/b'$. We claim that, in this case, $G=\langle [A(-1)], [C(b_0)]\rangle$. In fact, for an element $[B]\in G\setminus \langle [A(-1)]\rangle$, we have seen that $[B]=[C(b_0')]$ for some $b_0'\ne 0$. Hence $[A(b_0'')]=[B]\cdot [A']\in G$, where $b_0''=b_0/b_0'$, is an element of type 1 in $G$. In the same way as before we conclude that $[A(b_0'')]=[B]\cdot [C(b_0)]$ belongs to the group generated by $[A]=[A(-1)]$, hence $[B]$ belongs to the group $\langle [A(-1)], [C(b_0)]\rangle$. This shows that $G=\langle [A(-1)], [C(b_0)]\rangle$. Clearly $$[A(-1)]^2=[C(b_0)]^2=[I]\quad\text{and}\quad [A(-1)]\cdot [C(b_0)]=[C(b_0)]\cdot [A(-1)].$$ Therefore, $G$ is group of order $4$, a contradiction with our hypothesis.
\end{enumerate}

\item  $t=2$; in this case, $A=\mathcal B=\left( \begin{array}{cc}1&0\\ 1&1\end{array}\right)$. We have that
$$\sigma(A', A)=\left( \begin{array}{cc}a'-b'&b'\\ -(b'+d'-a'-c')& b'+d'\end{array}\right).$$
From Eq.~\eqref{LD}, we obtain $\sigma(A', A)=\varepsilon A'$. We split into cases:
\begin{enumerate}[$\bullet$]
\item {\bf Case b-1.} $\varepsilon=-1$. We obtain $$2a'-b'=b'+d'-a'-2c'=2b'=2d'+b'=0.$$  If $q$ is odd, it follows that $a'=b'=c'=d'=0$, a contradiction. If $q$ is even we get $b'=0$ and $a'=d'$. Since $[A']\ne [A]$, we have $c'\ne a'$ and we can easily check that $H=\langle [A], [A']\rangle$ is a group of order $p^2$. From Theorem~\ref{thm:p-group}, it follows that there do not exist $H$-invariants. However, since $H\le G$, it follows that there do not exist $G$-invariants.

\item {\bf Case b-2.} $\varepsilon=1$. We obtain $b'=a'-d'-b'=0$, hence $b'=0$ and $a'=d'$. We proceed in the same way as in \emph{Case b-1}, and conclude that there do not exist $G$-invariants.

\end{enumerate}

\item $t=3$; in this case, $A=\left( \begin{array}{cc}0&1\\ b&0\end{array}\right)$ for some non square $b$ in $\F_q^*$. We have
$$\sigma(A', A)=\left( \begin{array}{cc}-bd'&-c'\\ -b^2b'& -ba'\end{array}\right).$$
From Eq. \eqref{LD} we obtain $\sigma(A', A)=\varepsilon bA'$, hence 

\begin{equation}\begin{aligned}\label{involutions} a'  &= -\varepsilon d' \\
 c' &= -\varepsilon bb'.\end{aligned}\end{equation}

If $d'=0$, then $a'=0$ and $[A']=[A]$ or $[A']=[A]\cdot [A(-1)]$. Since we have supposed that $[A']$ is not in the group generated by $[A]$, this implies $[A']=[A]\cdot [A(-1)]$ and then $[A(-1)]$ is in $G$, a contradiction since $G$ is of type $3$. Hence $d'$ (and then $a'$) is nonzero. Also, if $b'=0$, then $c'=0$ and $A'$ is either the identity or an element of type 1, a contradiction since $G$ is of type $3$ and $[A']$ is not in the group generated by $[A]$. In particular, the elements $a', b', c'$ and $d'$ are nonzero. Therefore, from Eq.~\eqref{involutions}, there exists $t\in \F_q^*$ such that
$$[A']=\left[\left( \begin{array}{cc}t&-\varepsilon\\ b&-\varepsilon t\end{array}\right)\right].$$

For each $s, t\in \F_q^*$, set $$E_s=\left( \begin{array}{cc}s&1\\ b&s\end{array}\right) \quad \text{and}\quad F_t=\left( \begin{array}{cc}t&-1\\ b&-t\end{array}\right).$$

We have shown that, if $\C_{G}^*\ne \emptyset$, then $G\setminus \langle[A]\rangle$ contains only elements of the forms $[E_s]$ and $[F_t]$. If there exist $s$ and $t$ nonzero elements of $\F_q$ such that $[E_s], [F_t]\in G$, then for $f\in \C_G^*$ we have
$[E_s]\circ f=[F_t]\circ f=f$. Therefore, from Proposition~\ref{prop:sigma}, we have that $\sigma(E_s, F_t)=\delta (b-t^2) E_s$ for some $\delta\in \{-1, 1\}$. A simple calculation yields
$$\sigma(E_s, F_t)=(t^2-b)\cdot \left( \begin{array}{cc}-s&1\\ b&-s\end{array}\right)=(t^2-b)\cdot E_{-s}.$$

In particular, we obtain $E_{-s}=\delta'\cdot E_{s}$ for some $\delta'\in \{-1, 1\}$. For $\delta'=-1$, it follows that $2b=0$ but, since $G$ is of type $3$, we are in odd characteristic and then $b=0$, a contradiction. Hence $\delta'=1$ and then $s=0$, a contradiction with the fact that $s$ is nonzero.

This shows that $G\setminus \langle[A]\rangle$ only contains elements of the form $[E_s]$ or only contains elements of the form $[F_t]$. We split into cases:

\begin{itemize}
\item {\bf Case c-1.} $G\setminus \langle[A]\rangle$ only contains elements of the form $[F_t]$. 

Let $u\ne 0$ be such that $[F_{u}]\in G\setminus \langle[A]\rangle$. We claim that, in this case, $G=\langle[A], [F_{u}]\rangle$. In fact, if there exists $v\ne 0$ such that $[F_v]\in G\setminus \langle[A], [F_u]\rangle$, then $[F_u]\cdot [F_v]\in G\setminus \langle[A]\rangle$. A direct calculation yields

$$[F_u]\cdot [F_v]=[E_{\ell}],$$

where $\ell=\frac{uv-b}{v-u}$. But we have seen that $G\setminus \langle[A]\rangle$ cannot contain elements of both forms $[E_s]$ and $[F_t]$ with $s, t\ne 0$. This implies $\ell=0$, and then $[E_{\ell}]=[A]$, i.e., $[F_v]\in \langle[A], [F_u]\rangle$, a contradiction. 
Hence $G=\langle[A], [F_{u}]\rangle$. Clearly $[F_u]^2=[A]^2=I$ and $[F_u]\cdot [A]=[A]\cdot [F_u]$, and then $G$ is group of order $4$, a contradiction with our hypothesis.

\item {\bf Case c-2.} $G\setminus \langle[A]\rangle$ only contains elements of the form $[E_s]$. 

Notice that, extending our definition of $[E_s]$ to $s=0$, we get $[A]=[E_0]$. Also, we observe that $E_s=A+sI$, $A^2=bI$ and

$$[E_s]\cdot [E_t]=\begin{cases}[I]&\text{if}\;{s+t=0,}\\ 

[A_{f(s, t)}], f(s, t)=\frac{st+b}{s+t}&\text{otherwise}.\end{cases}$$

In particular, $H=\{[E_s], s\in \F_q\}\cup \{[I]\}$ is a group of order $q+1$ and $G\le H$. But, since $b$ is not a square in $\F_q$, $x^2-b\in \F_q[x]$ is irreducible and $K=\F_q[x]/(x^2-b)$ is the finite field with $q^2$ elements.. Let $k\alpha+w$ be a primitive element of $K=\F_{q^2}$, where $\alpha$ is a root of $x^2-b$ and $k, w\in \F_q$ with $k\ne 0$. We claim that, in this case, $H$ is the cyclic group generated by $[E_{w/k}]$. In fact, we have $[E_{w/k}]\in H$ and, for any $r\in \F_q$, we observe that $\alpha+r=(a\alpha+b)^m$ for some $m\in \N$. Therefore, $$(kx+w)^m\equiv x+r\pmod{x^2-b},$$ and then $(kA+wI)^r=A+rI=E_r$. Taking equivalence classes in $\PGL_2(\F_q)$ we obtain $[E_{w/k}]^m=[E_r]$, hence $H$ is cyclic. Since $G$ is a subgroup of $H$, it follows that $G$ is also cyclic and we get a contradiction.

\end{itemize}

\item $t=4$; in this case, $A=\left( \begin{array}{cc}0&1\\ c&1\end{array}\right)$ for some $c\in\F_q^*$ such that $x^2-x-c\in \F_q[x]$ is irreducible. We obtain
$$\sigma(A', A)=\left( \begin{array}{cc}c'-cd'&-c'\\ c'+ca'-cd'-c^2b'& -c'-ca'\end{array}\right).$$
From Eq.~\eqref{LD}, we obtain $\sigma(A', A)=\varepsilon cA'$ for some $\varepsilon\in \{\pm 1\}$. Therefore, $$\varepsilon ca'+cd'-c'=\varepsilon cb'+c'=\varepsilon cc'-c'-ca'+cd'+c^2b'=\varepsilon cd'+c'+ca'=0,$$
and, from $c'=-\varepsilon cb'$ and $c\ne 0$, it follows that
$$\varepsilon a'+d'+\varepsilon b'=-a'+d'+\varepsilon b'=\varepsilon d'-\varepsilon b'+a'=0.$$

From $\varepsilon a'+d'+\varepsilon b'=\varepsilon d'-\varepsilon b'+a'=0$ and $\varepsilon^2=1$, we conclude that $a'+b'+\varepsilon d'=d'+\varepsilon a'-b'$ and so $2b'=(\varepsilon -1)(a'-d')$. From $-a'+d'+\varepsilon b'=0$, we obtain $a'-d'=\varepsilon b'$ and so
$$2b'=(\varepsilon -1)(a'-d')=(\varepsilon -1)\varepsilon b'=(\varepsilon ^2-\varepsilon)b'.$$

We observe that $b'\ne 0$. For if $b'=0$, we have that $c'=0$ and $a'=d'$, hence $[A']$ is the identity element of $\PGL_2(\F_q)$, a contradiction with $[A']\in G\setminus \langle [A]\rangle$. If $q$ is even, $-1=1$ and so we have $\varepsilon=-1$. If $q$ is odd, since $b'\ne 0$ and $2b'=(\varepsilon ^2-\varepsilon)b'$, it follows that $\varepsilon^2-\varepsilon=2$ and so $\varepsilon=-1$. In particular, we have shown that $\varepsilon=-1$ and $b'\ne 0$. From the equalities $c'=-\varepsilon cb'=cb'$ and $a'-d'=\varepsilon b'=-b'$, it follows that $b'=t$, $c'=ct$, $a'=u$ and $d'=a'+b'=t+u$ for some $t\in \F_q^*$ and $u\in \F_q$. Therefore, there exists $s \in \F_q$ such that
$$[A']=\left[\left( \begin{array}{cc}s&1\\ c&s+1\end{array}\right)\right].$$
For $s\in \F_q$, set $A_s:=\left(\begin{array}{cc}s&1\\ c&s+1\end{array}\right)=A+sI$. In particular, we have shown that the elements of $G\setminus \langle [A]\rangle$ are of the form $[A_s]$ for some $s\in \F_q$ (in fact, $[A]=[A_0]$). We observe that $A^2=A+cI$ and then, for any $s, t\in \F_q$, the following holds:
$$[A_s]\cdot [A_t]=\begin{cases}[I]&\text{if}\;{s+t+1=0,}\\ 
[A_{g(s, t)}], g(s, t)=\frac{st+c}{s+t+1}&\text{otherwise}.\end{cases}$$

Therefore, the set $H=\{[A_s]\,;\, s\in \F_q\}\cup \{[I]\}$ is a group of order $q+1$ and $G\le H$. However, since $x^2-x-c\in \F_q[x]$ is irreducible, $K=\F_q[x]/(x^2-x-c)$ is the finite field with $q^2$ elements. Let $a\alpha+b$ be a primitive element of $K=\F_{q^2}$, where $\alpha$ is a root of $x^2-x-c$ and $a, b\in \F_q$ with $a\ne 0$. We claim that, in this case, $H$ is the cyclic group generated by $[A_{b/a}]$. In fact, we have $[A_{b/a}]\in H$ and, for any $s\in \F_q$, $\alpha+s=(a\alpha+b)^r$ for some $r\in \N$. Hence $$(ax+b)^r\equiv x+s\pmod{x^2-x-c},$$ and then $(aA+bI)^r=A+sI=A_s$. Taking equivalence classes in $\PGL_2(\F_q)$ we obtain $[A_{b/a}]^r=[A_s]$, hence $H$ is cyclic. Since $G$ is a subgroup of $H$, it follows that $G$ is also cyclic and we get a contradiction. 

\end{enumerate}
\end{proof}
For groups of order $4$, we use a simpler argument.
\begin{prop}\label{prop:order4}
Suppose that $G$ is a noncyclic subgroup of $\PGL_2(\F_q)$ of order $4$. Then no irreducible polynomial of degree at least $3$ is $G$-invariant.
\end{prop}
\begin{proof}
It is straightforward to check that any noncyclic group of order $4$ is generated by $2$ elements whose orders equal to $2$. Let $[A_1]$ and $[A_2]$ be elements of $\PGL_2(\F_q)$ such that $G=\langle [A_1], [A_2] \rangle$. Suppose that $f$ is a polynomial of degree at least $n\ge 3$ such that $[A_1]\circ f=[A_2]\circ f=f$. From Theorem~\ref{thm:ST12-3.3}, $n$ is divisible by $2$, i.e., $n=2m$ for some positive integer $m\ge 2$. Also, from Lemma~\ref{ST4.5}, it follows that $f$ divides both $F_{A_1, m}$ and $F_{A_2, m}$. However, from Lemma~\ref{lem:aux-divisor}, this implies that $[A_1]=[A_2]$, a contradiction with the fact that $G$ is not cyclic.
\end{proof}

It is straightforward to see that Theorem~\ref{thm:chap6-main-3} follows from Propositions~\ref{prop:order>4} and~\ref{prop:order4}. 

\subsection{A remark on quadratic invariants}
We observe that, if $f\in \F_q[x]$ is a quadratic irreducible polynomial and $H$ is a subgroup of $\PGL_2(\F_q)$ generated by elements $[B_1], \ldots, [B_s]$, $f$ is $H$-invariant if and only if $B_i\circ f=\lambda_i \cdot f$ for some $\lambda_i\in \F_q$; if we write $f=x^2+ax+b$, the equalities $B_i\circ f=\lambda_i\cdot f$ yield a system of $s$ polynomial equations, where $a, b$ and the $\lambda_i$'s are variables. The coefficients of these equations arise from the entries of the elements $B_i$. Given the elements $B_i$, we can easily solve this system.

\section{Rational functions and $[A]$-invariants}
In~\cite{MP17}, the authors show that the irreducible polynomials of degree $2n$ fixed by some involutions $[A]\in \PGL_2(\F_q)$ (i.e.$[A]^2=[I]$) arise from polynomials of the form $h^nf\left(\frac{g}{h}\right)$, where $f$ is irreducible of degree $n$ and $g, h$ are relatively prime polynomials such that $\max(\deg g, \deg h)=2$. They called $h^nf\left(\frac{g}{h}\right)$ as the \emph{quadratic transformation} of $f$ by $\frac{g}{h}$. This result generalizes the characterization $f(x)=x^ng\left(x+x^{-1}\right)$ for irreducible \emph{self-reciprocal polynomials} of even degree. In the same paper, the authors also explore some very particular elements $[A]\in \PGL_2(\F_q)$ of orders $3$ and $4$, where some \emph{cubic} and \emph{quartic} transformations appear. The aim of this section is to provide a full generalization of this result. We state our main result as follows.

\begin{theorem}\label{thm:chap6-main-2}
Let $[A]\in \PGL_2(\F_q)$ be an element of order $D=\ord([A])$. Then there exists a rational function $Q_{A}(x)=\frac{g_A(x)}{h_A(x)}$ of degree $D$ with the property that the $[A]$-invariants of degree $Dm>2$ are exactly the monic irreducible polynomials of the form $$F^{Q_A}:=(h_A)^m\cdot F\left(\frac{g_A}{h_A}\right),$$  
where $F$ has degree $m$. In addition, $Q_A$ can be explicitly computed from $A$.
\end{theorem}

We divide the proof of this result in two main cases. We first consider $[A]$ an element of type $t\le 3$: these cases can be easily covered by the works in~\cite{MP17} and~\cite{LR17}. For the case $t=4$, we generalize the main technique employed in~\cite{MP17}. In every case, we first consider an element of type $t$ in reduced form and then we apply Lemma~\ref{lem:conj-inv}.

\subsection{Elements of type $t\le 3$}\label{subsec:type-1-2-3}
We start considering elements in reduced form. This corresponds to the elements $[A(a)], [\mathcal E]$ and $[C(b)]$ in Theorem~\ref{thm:types}. For $[A(a)]$, we observe that $A(a)\circ f=f(ax)$ and $A(a)\circ f=\lambda \cdot f$ for some $\lambda\in \F_q^*$ yields $f(ax)=\lambda f(x)$: since $f$ has degree at least two, $f$ is not divisible by $x$ and so a comparison on the constant term in the previous equality yields $\lambda=1$. Therefore, $[A(a)]\circ f=f$ if and only if $A(a)\circ f=f$, i.e., $f(ax)=f(x)$. In particular, from Theorem~3.1 of~\cite{LR17}, the following corollary is straightforward.

\begin{cor}\label{cor:type-1}
Let $[A]$ be an element of type $1$ in reduced form such that $[A]$ has order $k$. For a monic irreducible polynomial of degree $km>2$, $[A]\circ f=f$ if and only if $f$ is a monic irreducible polynomial of the form $g(x^k)$ for some polynomial $g\in \F_q[x]$ of degree $m$.
\end{cor}

For the case $[A]=[\mathcal E]$, we see that $[\mathcal E]\circ f=f$ if and only if $\mathcal E\circ f=f$, i.e., $f(x+1)=f(x)$ (see the proof of Theorem~\ref{thm:p-group}). In particular, from Theorem~2.5 of~\cite{LR17}, the following corollary is straightforward.

\begin{cor}\label{cor:type-2}
For $f\in \F_q[x]$, $f(x+1)=f(x)$ if and only if $f(x)=g(x^p-x)$ for some $g\in \F_q[x]$. In particular, $f$ is a monic irreducible polynomial of degree $pm$ such that $[\mathcal E]\circ f=f$ if and only if $f$ is a monic irreducible polynomial of the form $g(x^p-x)$ for some polynomial $g\in \F_q[x]$ of degree $m$.
\end{cor}

We finally arrive in the elements of type $3$. We recall that any irreducible polynomial $f$ of degree at least $3$ that is fixed by $[A]\in \PGL_2(\F_q)$ has degree divisible by $D$, the order of $[A]$. In particular, any monic irreducible polynomial $f$ of degree at least $3$ such that $[C(b)]\circ f=f$ has degree divisible by $\ord([C(b)])=2$, say $\deg f=2m$. Additionally, from definition, $[C(b)]\circ f=f$ if and only if $$x^{2m}f\left(\frac{b}{x}\right)=\lambda\cdot f(x),$$ 
for some $\lambda\in \F_q^*$. Since $[C(b)]$ is of type $3$, $b$ is not a square of $\F_q$. In particular, the polynomial $x^2-b$ is irreducible over $\F_q$. If $\theta\in \F_{q^2}$ is a root of such polynomial, evaluating both sides of the previous equality at $x=\theta$, we obtain $b^mf(\theta)=\lambda f(\theta)$.
Since $f$ has degree at least three, $f$ is not divisible by $x^2-b$ and so $f(\theta)\ne 0$. Therefore, $\lambda=b^m$. According to Lemma~3 of~\cite{MP17}, we have that $x^{2m}f(b/x)=b^mf(x)$ if and only if $f(x)=x^{m}g(x+b/x)$ for some $g\in \F_q[x]$ of degree $m$ and so we obtain the following result.

\begin{lemma}\label{lem:type-3}
Let $f$ be a monic irreducible polynomial of degree $2m>3$ and let $b\in \F_q^*$ be a non square. Then $[C(b)]\circ f=f$ if and only if $f$ is a monic irreducible polynomial of the form $x^mg(x+b/x)$ for some $g\in \F_q[x]$ of degree $m$.
\end{lemma}

Combining the previous results, we obtain the following theorem.

\begin{theorem}\label{thm:types1-3}
Let $[A]\in \PGL_2(\F_q)$ be an element of type $t\le 3$ and order $D$. Then there exists a rational function $Q_A=\frac{g_A(x)}{h_A(x)}$ of degree $D$ with the property that, for any monic irreducible polynomial $f$ of degree $Dm>2$, $[A]\circ f=f$ if and only if $f$ is a monic irreducible of the form
$$F^{Q_A}=(h_A)^m\cdot F\left(\frac{g_A}{h_A}\right),$$
for some $F\in \F_q[x]$ of degree $m$. 
\end{theorem}
\begin{proof}
From Theorem~\ref{thm:types}, there exists an element $[\mathcal A]$ of type $t$ in reduced form and $[P]\in \PGL_2(\F_q)$ such that $[A]=[P]\cdot [\mathcal A]\cdot [P]^{-1}$. Let $f$ be a monic irreducible polynomial of degree $Dm>2$. From Lemma~\ref{lem:conj-inv}, $[A]\circ f=f$ if and only if $[\mathcal A]\circ g=g$, where $g=[P]^{-1}\circ f$. Write $P=\left(\begin{matrix}
w_1&w_3\\ w_2&w_4
\end{matrix}\right)$. We have three cases to consider.
\begin{enumerate}[(i)]
\item $t=1$; in this case, $[\mathcal A]=[A(a)]$ for some $a\in \F_q^*$ with multiplicative order $D$ and Corollary~\ref{cor:type-1} entails that $[A(a)]\circ g=g$ if and only if $g(x)=G(x^D)$ for some $G\in \F_q[x]$ of degree $m$. Therefore, $[A]\circ f=f$ if and only if $f=[P]\circ G(x^D)$. Now we just need to compute $[P]\circ G(x^D)$. If we write $G(x)=\sum_{i=0}^{m}a_ix^i$, from definition, we have that
$$[P]\circ G(x^D)=\lambda\cdot (w_3x+w_4)^{Dm}\cdot \sum_{i=0}^{m}a_i\left(\frac{w_1x+w_2}{w_3x+w_4}\right)^{Di}=\lambda\cdot \sum_{i=0}^{m}a_i(w_1x+w_2)^{Di}(w_3x+w_4)^{D(m-i)},$$
for some $\lambda\in \F_q^*$. The latter equality can be rewritten as
$$[P]\circ G(x^D)=h_A^m\cdot F\left(\frac{g_A}{h_A}\right)=F^{Q_A},$$
where $g_A=(w_1x+w_2)^D$, $h_A=(w_3x+w^4)^D$, $Q_A=\frac{g_A}{h_A}$ and $F=\lambda\cdot G\in \F_q[x]$ is a polynomial of degree $m$. Since $P\in \GL_2(\F_q)$, $w_1x+w_2$ and $w_3x+w_4$ are relatively prime, hence $g_A$ and $h_A$ are relatively prime. Therefore, $Q_A$ is a rational function of degree $D$.

\item $t=2$; in this case, $D=p$, $[\mathcal A]=[\mathcal E]$ and Corollary~\ref{cor:type-2} entails that $[\mathcal E]\circ g=g$ if and only if $g(x)=G(x^p-x)$ for some $G\in \F_q[x]$ of degree $m$. In the same way as before, we conclude that $[A]\circ f=f$ if and only if $f$ equals $F^{Q_A}$, where $F=\lambda \cdot G$, $Q_A=\frac{g_A}{h_A}$, $$g_A=(w_1x+w_2)^p-(w_1x+w_2)(w_3x+w_4)^{p-1}\quad \text{and}\quad h_A=(w_3x+w_4)^p.$$ Since $P\in \GL_2(\F_q)$, $w_1x+w_2$ and $w_3x+w_4$ are relatively prime, hence $g_A$ and $h_A$ are relatively prime. Therefore, $Q_A$ is a rational function of degree $p=D$.

\item $t=3$; in this case, $D=2$, $[\mathcal A]=[C(b)]$ (for some non square $b\in \F_q^*$) and Lemma~\ref{lem:type-3} entails that $[C(b)]\circ g=g$ if and only if $g(x)=x^mG(x+b/x)$ for some $G\in \F_q[x]$ of degree $m$. In the same way as before, we conclude that $[A]\circ f=f$ if and only if $f$ equals $F^{Q_A}$, where $F=\lambda \cdot G\in \F_q[x]$, $Q_A=\frac{g_A}{h_A}$, $$g_A=(w_1x+w_2)^2+(w_3x+w_4)^{2}\quad\text{and}\quad h_A=(w_1x+w_2)(w_3x+w_4).$$ Since $P\in \GL_2(\F_q)$, $w_1x+w_2$ and $w_3x+w_4$ are relatively prime and so $g_A$ and $h_A$ are relatively prime. Therefore, $Q_A$ is a rational function of degree $D=2$.
\end{enumerate}
\end{proof}

\subsection{Elements of type $4$}
As pointed out earlier, in~\cite{MP17}, the authors prove that two specific elements $[A]\in \PGL_2(\F_q)$ of order 3 and 4 are such that the $[A]$-invariants arise from \emph{cubic} and \emph{quartic} transformations, respectively (see Lemmas~11 and~12 of~\cite{MP17}) . Their idea relies on considering the field $K^{A}\subseteq K$ of the elements that are fixed by the automorphism $\Gamma_A: x\mapsto \frac{ax+b}{cx+d}$ of $K:=\F_q(x)$ associated to $[A]\in \PGL_2(\F_q)$ with $A=\left(\begin{matrix}
a&b\\ c&d
\end{matrix}\right)$: if the fixed field equals $\F_q(Q_A)$, $Q_A\in K$ is the rational function that yields the $[A]$-invariants. In that paper, they show that $Q_A$ can be taken as the sums of the elements lying in the \emph{orbit} of $x$ by the automorphism $x\mapsto \frac{ax+c}{bx+d}$, i.e., the \emph{trace} of $x$ in the extension $K/K^A$. Perhaps, the main obstruction appears when $A$ is a generic element. In fact, the trace function does not work in general: for instance, if $A=\mathcal E$, $\Gamma_A:x\mapsto x+1$ and so the trace equals $\sum_{i=0}^{p-1}(x+i)=e$, where $e=1$ if $p=2$ and $e=0$ if $p>2$. One can see that, in fact, the rational functions associated to elements of type $1$ and $2$ in reduced form (which are $x^k$ and $x^p-x$, respectively) are simply the \emph{product} of the elements lying in the orbit of $x$ by $\Gamma_A$, i.e., the \emph{norm} of $x$ in the extension $K/K^A$.

In this subsection, we refine this method to elements $[A]$ of type $4$ in reduced form, obtaining an explicit description of the fixed field $K^A$. In particular, we can avoid the construction via the trace map. We summarize some basic properties on the automorphisms $\Gamma_A$.

\begin{theorem}\label{thm:automorphism}
For $A, B\in \GL_2(\F_q)$, the following hold:
\begin{enumerate}[(a)]
\item for $[A]=[I]$, $\Gamma_A$ is the identity map,
\item the ordinary composition $\Gamma_A\circ \Gamma_B$ equals $\Gamma_{AB}$,
\item the automorphism $\Gamma_A$ is of finite order and its order coincides with the order of $[A]$ in $\PGL_2(\F_q)$.
\end{enumerate}
\end{theorem}
Clearly, $\Gamma_A=\Gamma_B$ if and only if $[A]=[B]$ in $\PGL_2(\F_q)$, hence $\mathrm{Aut}(K)$ is isomorphic to $\PGL_2(\F_q)$. The following result is classic.

\begin{theorem}\label{thm:fixed-field}
Let $G$ be any subgroup of $\mathrm{Aut}(K)$ of order $d=|G|$. Then  $K$ is a $d$-degree extension of the fixed field $K^G\subseteq K$ of $K$ by $G$. In particular, if $Q\in K^G$ is any rational function of degree $d$, $K^G=K(Q)$.
\end{theorem}

In particular, in order to obtain the fixed field $K^A$, we just need to find rational function $Q_A$ of degree $D=\ord([A])$ that is contained in $K^A$. When $A$ is an element of type 4 in reduced form, we can explicitly find such a $Q_A$.

\begin{prop}\label{prop:type-4-rational}
Let $A=D(c)$ be an element of type $4$ and let $D$ be the order of $[A]$ in $\PGL_2(\F_q)$. Additionally, let $\theta\in \F_{q^2}\setminus \F_q$ be a root of $x^2-x-c$ (i.e., $\theta$ is an \emph{eigenvalue} of $D(c)$). For $$g_c(x):=\frac{\theta^q(x+\theta^q)^D-\theta(x+\theta)^D}{\theta^q-\theta}\quad\text{and}\quad h_c(x):=\frac{(x+\theta^q)^D-(x+\theta)^D}{\theta^q-\theta},$$ the following hold:

\begin{enumerate}[(i)]
\item $g_c, h_c\in \F_q[x]$ are relatively prime polynomials of degree $D$ and $D-1$, respectively,

\item the roots of $h_c(x)$ lie in $\F_{q^2}$,

\item $g_c\left(\frac{c}{x+1}\right)=\theta^D\cdot \frac{g_c(x)}{(x+1)^D}$ and $h_c\left(\frac{c}{x+1}\right)=\theta^{D}\frac{h_c(x)}{(x+1)^D}$.
\end{enumerate}

In particular, the fixed field $K^A$ equals $\F_q(Q_c)$, where $Q_c=\frac{g_c}{h_c}$.
\end{prop}

\begin{proof}
\begin{enumerate}[(i)]
\item It is straightforward to check that the coefficients of $g_c$ and $h_c$ are invariant by the Frobenius map $a\mapsto a^q$ and so $g_c, h_c\in \F_q[x]$. Additionally, $g_c$ has degree $D$ (in fact, is a monic polynomial) and $h_c$ has degree $D-1$. We observe that $g_c-\theta h_c=(x+\theta^q)^D$ and $g_c-\theta^q h_c=(x+\theta)^D$, where $\theta\ne \theta^{q}$ (recall that $\theta\not\in \F_q$). In particular, the polynomials $g_c, h_c$ are relatively prime. 

\item Let $\alpha\in \overline{\F}_q$ be a root of $h_c$. Therefore, $(\alpha+\theta^q)^D=(\alpha+\theta)^D$. Clearly $\alpha\ne -\theta, -\theta^q$, since $\theta^q\ne -\theta$. In particular, $\left(\frac{\alpha+\theta^q}{\alpha+\theta}\right)^D=1$. Recall that $D$ divides $q+1$ (see Lemma~\ref{lem:order}) and so $\frac{\alpha+\theta^q}{\alpha+\theta}=\delta$ for some $\delta\in \F_{q^2}$. Since $\theta^q\ne \theta$, $\delta\ne 1$ and then $\alpha=\frac{\theta^q-\theta\delta}{\delta-1}\in \F_{q^2}$.

\item Since $\theta$ is a root of $x^2-x-c$, we have that $\theta^q+\theta=1$ and $\theta^{q+1}=-c$. We just compute $g_c\left(\frac{c}{x+1}\right)$ since the computation of $h_c\left(\frac{c}{x+1}\right)$ is quite similar. We observe that
$$\frac{c}{x+1}+\theta=\theta\cdot \frac{x+1-\theta^q}{x+1}=\theta\cdot \frac{x+\theta}{x+1},$$
and so $\frac{c}{x+1}+\theta^{q}=\theta^{q}\cdot \frac{x+\theta^q}{x+1}$. Therefore, 
$$g_c\left(\frac{c}{x+1}\right)=\frac{\theta^q\cdot \theta^{qD}(x+\theta^q)^D-\theta\cdot \theta^D(x+\theta)^D}{(\theta^q-\theta)(x+1)^D}=\theta^D\cdot \frac{g_c(x)}{(x+1)^D},$$
since $\theta^{qD}=\theta^D$ (see the proof of~Lemma~\ref{lem:order}).
\end{enumerate}

\noindent It is clear that the the automorphism $\Gamma_A$ induced by $A=D(c)$ is given by $x\mapsto \frac{c}{x+1}$. In particular, its order equals $D$. From item (iii), one can see that $Q_c=\frac{g_c}{h_c}$ is fixed by this automorphism and, from item (i), $Q_c$ is a rational function of degree $D$. Therefore, from Theorem~\ref{thm:fixed-field}, we have that $K^{A}=\F_q(Q_c)$.
\end{proof}

Let $f$ be an irreducible polynomial of degree $Dm>2$. From definition, $[D(c)]\circ f=f$ if and only if $$(x+1)^{Dm}f\left(\frac{c}{x+1}\right)=\lambda\cdot f(x),$$ 
for some $\lambda\in \F_q^*$. Since $[D(c)]$ is of type $4$, $x^2-x-c$ is irreducible over $\F_q$ (see~Theorem~\ref{thm:types}). If $\theta\in \F_{q^2}$ is a root of such polynomial, evaluating both sides of the previous equality at $x=-\theta$, we obtain $$(1-\theta)^{Dm}f\left(\frac{c}{1-\theta}\right)=\lambda\cdot f(-\theta).$$
Since $\theta^2-\theta=c=-\theta^{q+1}$ and $\theta^q=1-\theta$, we obtain $\frac{c}{1-\theta}=-\theta$ and so the latter is equivalent to $$\theta^{qDm}f(-\theta)=\lambda f(-\theta).$$ We observe that $-\theta$ has minimal polynomial $x^2+x-c$ and, since $f$ is irreducible of degree $Dm>2$, $f$ is not divisible by $x^2+x-c$. Therefore, $f(-\theta)\ne 0$ and so $\lambda=\theta^{qDm}=\theta^{Dm}$ (recall that $\theta^{qD}=\theta^{D}$ and so $\theta^D\in \F_q$). We obtain the following result.

\begin{theorem}\label{thm:auxiliar-type-4}
Let $A=D(c)$ be an element of type $4$ and order $D$. Additionally, let $\theta$ be a root of $x^2-x-c$. Let $g_c(x), h_c(x)$ and $Q_c(x)$ be defined as in Proposition~\ref{prop:type-4-rational}. A monic irreducible polynomial $f$ of degree $Dm>2$ is such that 
$(x+1)^{Dm}f\left(\frac{c}{x+1}\right)=\theta^{Dm}f(x)$ (that is, $[A]\circ f=f$) if and only if $f$ is a monic irreducible polynomial of the form
$$F^{Q_c}=(h_c)^mF\left(\frac{g_c}{h_c}\right),$$
for some $F\in \F_q[x]$ of degree $m$.
\end{theorem}
\begin{proof}
If $$\frac{f}{h_c^m}=F\left(\frac{g_c}{h_c}\right)=F(Q_c),$$
for some $F\in \F_q[x]$, $\frac{f}{h_c^m}$ is fixed by the automorphism $\Gamma_{A}: x\mapsto \frac{c}{x+1}$ (see Proposition~\ref{prop:type-4-rational}). We recall that $\Gamma_{A}(h_c)=\theta^D\cdot \frac{h_c}{(x+1)^D}$ and so $\theta^{Dm}f(x)=(x+1)^{Dm}f\left(\frac{c}{x+1}\right)$. Conversely, if $$\theta^{Dm}\cdot f(x)=(x+1)^{Dm}f\left(\frac{c}{x+1}\right),$$ by the previous calculations we obtain $\Gamma_{A}\left(\frac{f}{h_c^m}\right)=\frac{f}{h_c^m}$, i.e., $\frac{f}{h_c^m}$ is fixed by $\Gamma_A$. Additionally, the roots of $h_c(x)$ lie in $\F_{q^2}$ (see Proposition~\ref{prop:type-4-rational}). Since $f$ is irreducible of degree $Dm>2$, it follows that $f$ and $h_c$ are relatively prime and so $\frac{f}{h_c^m}$ is a rational function of degree $Dm$. Since the fixed field $K^A$ is generated by $Q_c$, it follows that
$$\frac{f}{h_c^m}=F(Q_c),$$
for some $F\in\F_q(x)$. Clearly, $F$ is a rational function of degree $m$. We shall prove that $F$ is, in fact, a polynomial. Observe that, if $F$ were not a polynomial, then it would have a pole at some element $\gamma\in \overline{\F}_q$. This implies that $\frac{f}{h_c^n}$ would have a pole at some root $\lambda\in \overline{\F}_q$ of $g_c-\gamma h_c$. However, since $f$ is a polynomial, the possible poles of $\frac{f}{h_c^m}$ are at the roots of $h_c$. In particular, $\lambda$ would be a root of $g_c-\gamma h_c$ and $h_c$. This contradicts the fact that $h_c$ and $g_c$ are relatively prime (see Proposition~\ref{prop:type-4-rational}). Therefore, $F$ is a polynomial and and so $f=(h_c)^mF(Q_c)=F^{Q_c}$.
\end{proof}

Following the proof of Theorem~\ref{thm:types1-3}, we extend the previous result to general elements of type 4.
\begin{prop}\label{prop:type-4!}
Let $[A]\in \PGL_2(\F_q)$ be an element of type $4$ and order $D$. Then there exists a rational function $Q_A=\frac{g_A(x)}{h_A(x)}$ of degree $D$ with the property that, for any monic irreducible polynomial $f$ of degree $Dm>2$, $[A]\circ f=f$ if and only if $f$ is a monic irreducible of the form
$$F^{Q_A}=(h_A)^m\cdot F\left(\frac{g_A}{h_A}\right),$$
for some $F\in \F_q[x]$ of degree $m$. 
\end{prop}
\begin{proof}
From Theorem~\ref{thm:types}, there exist an element $[D(c)]$ of type $4$ and $[P]\in \PGL_2(\F_q)$ such that $[A]=[P]\cdot [D(c)]\cdot [P]^{-1}$. Clearly, $[D(c)]$ also has order $D$. Let $f$ be a monic irreducible polynomial of degree $Dm>2$. From Lemma~\ref{lem:conj-inv}, $[A]\circ f=f$ if and only if $[D(c)]\circ g=g$, where $g=[P]^{-1}\circ f$. Write $P=\left(\begin{matrix}
w_1&w_3\\ w_2&w_4
\end{matrix}\right)$. Let $g_c, h_c$ and $Q_c$ be as in Proposition~\ref{prop:type-4-rational}. We observe that, from Theorem~\ref{thm:auxiliar-type-4}, $[D(c)]\circ g=g$ if and only if $g=(h_c)^mG(Q_c)=G^{Q_c}$ for some polynomial $G\in \F_q[x]$ of degree $m$. Therefore, $[A]\circ f=f$ if and only if $f=[P]\circ G^{Q_c}$. If we write $G(x)=\sum_{i=0}^{m}a_ix^i$, it follows that $G^{Q_c}=(h_c)^m\cdot G(Q_c)=\sum_{i=0}^ma_i(g_c)^i(h_c)^{m-i}$ and so
$$[P]^{-1}\circ G_{Q_c}=\lambda\cdot \sum_{i=0}^ma_i(G_A)^i(H_A)^{m-i},$$
where $\lambda \in \F_q^*$, $G_A=(w_3x+w_4)^D\cdot g_c\left(\frac{w_1x+w_2}{w_3x+w_4}\right)=P^{-1}\circ g_c$ and $$H_A=(w_3x+w_4)^D\cdot h_A\left(\frac{w_1x+w_2}{w_3x+w_4}\right)=(w_3x+w_4)\cdot (P^{-1}\circ h_c).$$

In other words, $f=(H_A)^m\cdot F(Q_A)=F^{Q_A}$, where $F=\lambda \cdot G$ has degree $m$ and $Q_A=\frac{G_A}{H_A}$. Now, we need to show that $Q_A$ is indeed a rational function of degree $D$. We observe that, from the definition of $G_A$ and $H_A$, at least one of these polynomials have degree $D$ if and only if $P^{-1}\circ g_c$ has degree $D$ or $P^{-1}\circ h_c$ has degree $D-1$. In other words, the compositions $P^{-1}\circ$ preserves the degree of at least one of the polynomials $g_c, h_c$. From item (i) of Lemma~\ref{lem:propertiesAf}, this does not occur if and only if $w_2\ne 0$ and both $g_c$ and $h_c$ vanish at $w_1/w_2\in \F_q$. However, we have seen that $g_c$ and $h_c$ are relatively prime and so they cannot have a common root. From now, we just need to prove that $G_A$ and $H_A$ are relatively prime. Since $P\in \GL_2(\F_q)$, $w_1x+w_2$ and $w_3x+w_4$ are relatively prime: expanding $G_A$, one can see that $w_3+w_4x$ and $G_A$ are relatively prime. So it suffices to prove that $G_A=P^{-1}\circ g_c$ and $H=P^{-1}\circ h_c$ are relatively prime. Since $h_c$ and $g_c$ are relatively prime, there exist polynomials $R_1$ and $R_2$ in $\F_q[x]$ such that $g_cR_1+h_cR_2=1$: in this case, $g_cR_1$ and $h_cR_2$ have the same degree $n>1$ and $g_cR_1\ne -h_cR_2$. From items (ii) and (v) of~Lemma~\ref{lem:propertiesAf}, it follows that $$R_3=P^{-1}\circ (g_cR_1)+P^{-1}\circ(h_cR_2)=G_A\cdot (P^{-1}\circ R_1)+H\cdot (P^{-1}\circ R_2),$$ divides $(w_3x+w_4)^{n-1}\cdot (A\circ (g_cR_1+h_cR_2))=(w_3x+w_4)^{n-1}\cdot(A\circ 1)=(w_3x+w_4)^{n-1}$. In particular, any common divisor of $H$ and $G_A$ divides $R_3$, hence it divides $(w_3x+w_4)^{n-1}$. Recall that $G_A$ is relatively prime with $w_3x+w_4$ and so we conclude that $G_A$ and $H$ are relatively prime.
\end{proof}
\subsection{Some examples}
We finish this section exemplifying the applicability of Theorem~\ref{thm:chap6-main-2} with $A=\left(\begin{matrix}
0&1\\ -1& 1
\end{matrix}\right)$, that turns out to be of three different types, according to the class of $q$ modulo $3$. We observe that $A^3=-I$ and so $[A]^3=[I]$. Hence, $[A]$ has order $D=3$. Its characteristic polynomial equals $x^2-x+1$. In particular, there exists a rational function $Q_A=\frac{g_A}{h_A}$ of degree $3$ such that the $[A]$ invariants of degree $3m$ are exactly the monic irreducible polynomials of the form $F^{Q_A}=(h_A)^mF\left(\frac{g_A}{h_A}\right)$, where $F\in \F_q[x]$ has degree $m$ and $Q_A$ is given as follows.

\begin{itemize}
\item For $q\equiv 1\pmod 3$, $x^2-x+1$ is the minimal polynomial of the primitive $d$-th roots of unity, where $d=6$ if $q$ is odd and $d=3$ if $q$ is even. In any case, $q\equiv 1\pmod d$ and so there exists $\theta\in \F_q$ such that $\theta^2-\theta+1=0$, i.e., $\theta$ is an eigenvalue of $A$. We have that $\theta, -\theta^2\in \F_q$ are the eigenvalues of $A$. In particular, $A$ is of type $1$ and $[A]=[P]\cdot [A(-\theta)]\cdot [P]^{-1}$, where $P=\left(\begin{matrix}
1&1\\ \theta& -\theta^2
\end{matrix}\right)$. Therefore, according to the proof of~Theorem~\ref{thm:types1-3}, we have that $$Q_A=\frac{(x+\theta)^3}{(x-\theta^2)^3}=\frac{x^3+3\theta x^2+3(\theta-1)x-1}{x^3-3(\theta-1)x^2-3\theta x-1}.$$

\item For $q\equiv 0\pmod 3$, $x^2-x+1=x^2+2x+1=(x+1)^2$ and so $A$ is of type $2$ and $[A]=[P]\cdot [\mathcal E]\cdot [P]^{-1}$, where $P=\left(\begin{matrix}
1&1\\ 1& -1
\end{matrix}\right)$. Therefore, according to the proof of~Theorem~\ref{thm:types1-3}, we have $$Q_A=\frac{(x+1)^3-(x+1)(x-1)^2}{(x-1)^3}=\frac{x^2+x}{x^3-1}.$$

\item For $q\equiv 2\pmod 3$, $x^2-x+1$ is the minimal polynomial of the primitive $d$-th roots of unity, where $d=6$ if $q$ is odd and $d=3$ if $q$ is even. In any case, $q\equiv 5\pmod d$ and so $x^2-x+1$ is irreducible over $\F_q$.  In fact, $A$ is of type $4$ in reduced form and we have $A=D(-1)$. Let $\theta, \theta^q$ be th eigenvalues of $A$. Therefore, $\theta^{q}+\theta=1$, $\theta^2=\theta-1$ and $\theta^{3}=-1$. According to Theorem~\ref{thm:auxiliar-type-4} and Proposition~\ref{prop:type-4-rational}, we have that $$Q_A=\frac{\theta^q(x+\theta^q)^3-\theta(x+\theta)^3}{(x+\theta^q)^3-(x+\theta)^3}=\frac{(\theta^q-\theta)(x^3+3x^2-1)}{(\theta^q-\theta)(3x^2+3x)}=\frac{x^3+3x^2-1}{3x^2+3x}.$$
\end{itemize}

\section{On the number of $[A]$-invariants}
In this section, we provide complete enumeration formulas for the number of $[A]$-invariants of degree $n>2$. Namely, we obtain the following result.
\begin{theorem}\label{thm:chap6-main-1}
Let $[A]\in \PGL_2(\F_q)$ be an element of order $D=\ord([A])$. Then, for any integer $n>2$, the number $\n_{A}(n)$ of $[A]$-invariants of degree $n$ is \emph{zero} if $n$ is not divisible by $D$ and, for $n=Dm$ with $m\in \N$, the following hold:
\begin{equation}\label{eq:invariants-formula}\n_{A}(Dm)=\frac{\varphi(D)}{Dm}\left(c_{A}+\sum_{d|m\atop \gcd(d, D)=1}\mu(d)(q^{m/d}+\eta_{A}(m/d))\right),\end{equation}
where $\varphi$ is the Euler Phi function, $\eta_{A}:\N\to \N$ and $c_{A}\in \Z$ are given as follows
\begin{enumerate}
\item $c_{A}=0$ and $\eta_{A}(t)=-1$ if $A$ is of type $1$,
\item $c_{A}=0$ and $\eta_{A}(t)=0$ if $A$ is of type $2$, 
\item $c_{A}=-1$ and $\eta_{A}(t)=0$ if $A$ is of type $3$,
\item $c_{A}=0$ and $\eta_{A}(t)=(-1)^{t+1}$ if $A$ is of type $4$.
\end{enumerate}
\end{theorem}

\begin{remark}
We emphasize that cases $1, 2$ and $3$ of the previous theorem are easily deduced from some well known results that are further referenced. Nevertheless, the case $3$ is far the most complicated and is our genuine contribution. We group them together just to provide a unified statement.
\end{remark}

We naturally study the number of $[A]$-invariants according to the type of $A$. From Theorem~\ref{thm:conjugates-main}, it suffices to consider elements of $\PGL_2(\F_q)$ of type $1\le t\le 4$ in reduced form. We start with elements of type $1$ and $2$. If $[A]$ has type $t=1, 2$ and is in reduced form, we see that $[A]\circ f$ corresponds to $f(ax)$ and $f(x+1)$ respectively. Additionally, we have seen that, in these cases, $[A]\circ f=f$ if and only if $A\circ f=f$. Hence we are looking for the monic irreducible polynomials of degree $n$ satisfying $f(x)=f(ax)$ or $f(x)=f(x+1)$. The polynomials satisfying such identities were previously explored by Garefalakis~\cite{Gar11}. From Theorems~2 and~4 of~\cite{Gar11}, we obtain the following result.

\begin{lemma}
Let $[A]\in \PGL_2(\F_q)$ be an element of type $t\le 2$ and order $D$. Then, for any integer $n>2$, the number $\n_{A}(n)$ of $[A]$-invariants of degree $n$ is \emph{zero} if $n$ is not divisible by $D$ and, for $n=Dm$ with $m\in \N$, the following hold:
\begin{equation*}\n_{A}(Dm)=\frac{\varphi(D)}{Dm}\sum_{d|m\atop \gcd(d, D)=1}\mu(d)(q^{m/d}-\varepsilon),\end{equation*}
where $\varepsilon=1$ if $t=1$ and $\varepsilon=0, D=p$ if $t=2$.
\end{lemma}

Recall that an element of type $3$ in reduced form equals $C(b)$ for some non square $b\in \F_q^*$ and $[C(b)]$ has order two. Also, for a monic irreducible polynomial of degree $2m>3$, we have seen that $[C(b)]\circ f=f$ if and only if $x^{2m}f\left(\frac{b}{x}\right)=b^m\cdot f(x)$ (see the comments after Corollary~\ref{cor:type-2}). In particular, from Corollary~7 of~\cite{MP17}, we obtain the following result.

\begin{lemma}
Let $[A]\in \PGL_2(\F_q)$ be an element of type $3$. In particular, its order is $D=2$. Then, for any integer $n>2$, the number $\n_{A}(n)$ of $[A]$-invariants of degree $n$ is \emph{zero} if $n$ is not divisible by $D$ (i.e., $n$ is odd) and, for $n=2m$ with $m\in \N$, the following hold:
\begin{equation*}\n_{A}(2m)=\frac{1}{2m}\left(-1+\sum_{d|m\atop \gcd(d, 2)=1}\mu(d)q^{m/d}\right).\end{equation*}
\end{lemma}
In particular, cases $1$, $2$ and $3$ of Theorem~\ref{thm:chap6-main-1} follow directly from the previous lemmas.

\subsection{Elements of type 4}
Here we establish the last case of Theorem~\ref{thm:chap6-main-1}, that corresponds to elements of type $4$. Again, one may consider only elements of type $4$ in reduced form. We emphasize that the previous enumeration formulas for elements of type $t\le 3$ are based in the \emph{Mobius Inversion Formula} and its generalizations. This inversion formula is often employed when considering the enumeration of irreducible polynomials with specified properties. We recall an interesting generalization of this result.

\begin{theorem}\label{mobius}
Let $\chi:\mathbb N\to \mathbb C$ be a completely multiplicative function (which is, in other words, an homomorphism between the monoids $(\mathbb N, +)$ and $(\mathbb C, \cdot)$). Also let $\mathcal L, \mathcal K:\mathbb N\to \mathbb C$ be two functions such that
$$\mathcal{L}(n)=\sum_{d|n}\chi(d)\cdot \mathcal{K}\left(\frac{n}{d}\right), n\in \mathbb N.$$
Then,
$$\mathcal{K}(n)=\sum_{d|n}\chi(d)\cdot\mu(d)\cdot \mathcal{L}\left(\frac{n}{d}\right), n\in \mathbb N.$$
\end{theorem}
An interesting class of completely multiplicative functions is the class of {\em Dirichlet Characters} and, for instance, the {\em principal Dirichlet character modulo $d$} is the function $\chi_{d}:\mathbb N\to \mathbb N$ such that $\chi_d(n)=1$ if $\gcd(d, n)=1$ and $\chi_d(n)=0$, otherwise. We first present a direct consequence of the results contained in~Subsection~\ref{subsec:lemma-aux}. 
\begin{lemma}\label{lem:type-4-enum-1}
Let $A$ be an element of $\GL_2(\F_q)$, let $D$ be the order of $[A]$. Then, for any $m\in \N$, the $[A]$-invariants of degree $Dm>2$  are exactly the irreducible factors of degree $Dm$ of $F_{A^j, m}$, where $j$ runs through the positive integers $\le D-1$ such that $\gcd(j, D)=1$.
\end{lemma}

\begin{proof}
According to Lemma~\ref{ST4.5}, the $[A]$-invariants of degree $Dm>2$ are exactly the irreducible factors of degree $Dm$ of $F_{A, \ell\cdot m}$, where $\ell$ runs through the positive integers $\le D-1$ such that $\gcd(\ell, D)=1$. Additionally, according to Lemma~\ref{power}, for each $\ell$ the following hold: if we set $j(\ell)$ as the least positive solution of $j\ell \equiv 1\pmod D$, the irreducible factors of degree $Dm$ of $F_{A, \ell\cdot m}$ are exactly the irreducible factors of degree $Dm$ of $F_{A^{j(\ell)}, m}$. Clearly $j(\ell)$ runs through the positive integers $j\le D-1$ such that $\gcd(j, D)=1$ (that is, $j(\ell)$ is a permutation of the numbers $\ell$).
\end{proof}

From now, it is sufficient to count the irreducible polynomials of degree $Dm$ that divide the polynomials $F_{A^j, m}$ for $j\le D-1$ and $\gcd(D, j)=1$. In this case, it is crucial to study the coefficients of $A^j$. When $A$ is an element of type $4$ in reduced form, we can obtain a complete description on the powers of $A$.

\begin{prop}\label{prop:powers-A}
If $A=D(c)$ is an element of type $4$, then \begin{equation}\label{eq:powers-A}A^j=\left( \begin{array}{cc}a_j&b_j\\ c_j & d_j\end{array}\right)=\delta \left( \begin{array}{cc}\alpha^{qj+1}-\alpha^{q+j} & \alpha^{q(j+1)+1}-\alpha^{q+j+1}\\
                \alpha^j-\alpha^{qj}& \alpha^{j+1}-\alpha^{q(j+1)}\end{array}\right), j\in \mathbb Z,\end{equation}
where $\alpha$ is an eigenvalue of $A$ and $\delta=(\alpha-\alpha^q)^{-1}$. In particular, if $D$ is the order of $[A]$, $c_j\ne 0$ for $1\le j\le D-1$.        
\end{prop}

\begin{proof}
Since $A$ is of type $4$, $A$ is a diagonalizable matrix over $\F_{q^2}$ but not over $\F_q$ and we can write
\[
A= M \left( \begin{array}{cc}\alpha&0\\ 0&\alpha^{q}\end{array}\right) M^{-1},
\quad\text{
where}
\quad
M = \left( \begin{array}{cc}\alpha^q& \alpha\\ -1& -1\end{array}\right)
\]
is an invertible matrix and $\alpha$ is an eigenvalue of $A$. From now, Eq.~\eqref{eq:powers-A} follows by direct calculations. We see that $c_j=0$ if and only if $\alpha^j=\alpha^{qj}$. The latter is equivalent to $[A]^j=[I]$ and so $j$ must be divisible by $D$. In particular, for $1\le j\le D-1$, $c_j\ne 0$. 
\end{proof}

From Lemma~\ref{ST4.5}, in general, the irreducible factors of $F_{A^j, m}$ have degree divisible by $D$. The problem relies on counting the irreducible polynomials of degree one and two. From the previous proposition, we describe the linear and quadratic irreducible factors of $F_{A^j, m}$ as follows.

\begin{lemma}\label{lem:type-4-enum-2}
Suppose that $A=D(c)$ is an element of type 4 and order $D$. For any positive integers $j$ and $m$ such that $j\le D-1$ and $\gcd(j, D)=1$, the polynomial $F_{A^j, m}\in \F_q[x]$ has degree $q^m+1$, is free of linear factors and has at most one irreducible factor of degree $2$. In addition, $F_{A^j, m}$ has an irreducible factor of degree $2$ if and only if $m$ is even and, in this case, this irreducible factor is $x^2+c^{-1}x-c^{-1}$.
\end{lemma}

\begin{proof}
From definition, $F_{A^j, m}=b_jx^{q^m+1}-a_jx^{q^m}+d_jx-c_j$. From Proposition~\ref{prop:powers-A}, $c_j\ne 0$ if $1\le j\le D-1$. Therefore, hence $F_{A^j, m}$ has degree $q^m+1$ if $j\le D-1$. We split the proof into cases, considering the linear and quadratic irreducible polynomials.
\begin{enumerate}
\item If $F_{A^j, m}$ has a linear factor, there exists $\gamma\in \F_q$ such that $F_{A^j, m}(\gamma)=0$. In this case, $\gamma^q=\gamma$ and a simple calculation yields $F_{A^j, m}(\gamma)=b_j\gamma^2+(d_j-a_j)\gamma-c_j$ and so $\gamma$ is a root of $p_j(x)=b_jx^2+(d_j-a_j)x-c_j$. Let $\alpha, \alpha^q$ be the eigenvalues of $A=D(c)$, hence $\alpha^q+\alpha=1$ and $\alpha^{q+1}=-c$. From Eq.~\eqref{eq:powers-A} and the previous equalities, we can easily deduce that $d_j-a_j=c_j$ and $b_j=cc_j$. Therefore, $p_j(x)$ equals $cx^2+x-1$ (up to a constant). This shows that $\gamma$ is a root of $x^2+c^{-1}x-c^{-1}$. However, since $A=D(c)$ is of type $4$, its characteristic polynomial $p(x)=x^2-x-c$ is irreducible over $\F_q$ and so is $x^2p(\frac{1}{x})=x^2+c^{-1}x-c^{-1}$. In particular, $\gamma$ cannot be an element of $\F_q$.
\item If $F_{A^j, m}$ has an irreducible factor of degree $2$, there exists $\gamma\in \F_{q^2}\setminus \F_q$ such that $F_{A^j, m}(\gamma)=0$. We observe that, in this case, $\gamma^{q^2}=\gamma$. For $m$ even, $F_{A^j, m}(\gamma)=b_j\gamma^2+(d_j-a_j)\gamma-c_j$ and in the same way as before we conclude that $x^2+c^{-1}x-c^{-1}$ is the only quadratic irreducible factor of $F_{A^j, m}$. If $m$ is odd, $\gamma^{q^m}=\gamma^q$ and equality $F_{A^j, m}(\gamma)=0$ yields $b_j\gamma^{q+1}-a_j\gamma^q+d_j\gamma-c_j=0$.
Raising the $q$-th power in the previous equality and observing that $\gamma^{q^2}=\gamma$, we obtain $$b_j\gamma^{q+1}-a_j\gamma+d_j\gamma^q-c_j=0,$$ and so $(\gamma^q-\gamma)(a_j+d_j)=0$. Since $\gamma$ is not in $\F_q$, the last equality implies that $a_j=-d_j$. However, from Eq.~\eqref{eq:powers-A}, we obtain $\alpha^{qj+1}-\alpha^{q+j}=\alpha^{q(j+1)}-\alpha^{j+1}$. Therefore, $(\alpha^q-\alpha)(\alpha^{qj}+\alpha^j)=0$. Recall that, since $A=D(c)$ is of type $4$, $\alpha$ is not in $\F_q$, i.e., $\alpha^q\ne \alpha$. Therefore $\alpha^{qj}=-\alpha^j$ and then $\alpha^{2qj}=\alpha^{2j}$. Again, from Eq.~\eqref{eq:powers-A}, this implies that $[A]^{2j}=1$, hence $2j$ is divisible by $D$. However, since $j$ and $D$ are relatively prime, it follows that $D$ divides $2$. This is a contradiction, since any element of type $4$ has order $D>2$ (see Lemma~\ref{lem:order}).
\end{enumerate}
\end{proof}

All in all, we finally add the enumeration formula for the number of $[A]$-invariants in the case when $[A]$ is of type $4$, completing the proof of Theorem~\ref{thm:chap6-main-1}.

\begin{theorem}
Suppose that $A$ is an element of type $4$ and set $D=\ord([A])$. Then $\n_A(n)=0$ if $n$ is not divisible by $D$ and, for $n=Dm$,
$$\n_A(Dm)=\frac{\varphi(D)}{Dm}\sum_{d|m\atop \gcd(d, D)=1}(q^{m/d}+\epsilon(m/d))\mu(d),$$ 
where $\epsilon(s)=(-1)^{s+1}$. 
\end{theorem}

\begin{proof}
From Theorem~\ref{thm:conjugates-main}, we can suppose that $A$ is in the reduced form, i.e., $A=D(c)$ for some $c$ such that $x^2-x-c$ is irreducible over $\F_q$.  For each positive integer $j$ such that $j\le D-1$ and $\gcd(j, D)=1$, let $n(j)$ be the number of irreducible factors of degree $Dm$ of $F_{A^j, m}$. From Lemma~\ref{lem:type-4-enum-1}, it follows that $$N_A(Dm)=\sum_{j\le D-1\atop \gcd(j, D)=1}n(j).$$
Fix $j$ such that $j\le D-1$ and $\gcd(j, D)=1$. According to Lemma~\ref{ST4.5}, the irreducible factors of $F_{A^j, m}$ are of degree $Dm$, of degree $Dk$, where $k$ divides $m$ and $\gcd(\frac{m}{k}, D)=1$ and of degree at most $2$. For each divisor $k$ of $m$ such that $\gcd(\frac{m}{k}, D)=1$, let $P_{k, m}$ be the product of all irreducible factors of degree $Dk$ of $F_{A^j, m}$ and let $\mathcal L(k)$ be the number of such irreducible factors. Also, set $$\varepsilon_m(x)=\gcd(F_{A^j, m}(x), x^2+c^{-1}x-c^{-1}).$$ Therefore, from Lemma~\ref{lem:type-4-enum-2}, we obtain the following identity
$$\frac{F_{A^j, m}}{\varepsilon_m(x)}=\prod_{k|m\atop \gcd(\frac{m}{k}, D)=1}P_{k, m}.$$
From Lemma~\ref{lem:type-4-enum-2}, $F_{A^j, m}$ has degree $q^m+1$ and the degree of $\varepsilon_m(x)$ is either $0$ or $2$, according to whether $m$ is odd or even. In particular, if we set $\epsilon(m)=(-1)^{m+1}$, taking degrees on the last equality we obtain:
$$q^m+1-\deg(\varepsilon_m(x))=q^m+\epsilon(m)=\sum_{k|m\atop \gcd(\frac{m}{k}, D)=1}\mathcal L (k)\cdot (kD)=\sum_{k|m}\mathcal L(k)\cdot (kD)\cdot \chi_{D}\left(\frac{m}{k}\right),$$
where $\chi_D$ is the {\em principal Dirichlet character modulo $D$}.
From Theorem~\ref{mobius}, we obtain
$$\mathcal L(k)\cdot kD=\sum_{d|k}(q^{k/d}+\epsilon(k/d))\cdot \mu(d)\cdot \chi_D(d)$$
for any $k\in \mathbb N$.
Hence $$\mathcal L(m)=\frac{1}{Dm}\sum_{d|m}(q^{m/d}+\epsilon(m/d))\cdot \mu(d)\cdot \chi_D(d)=\frac{1}{Dm}\sum_{d|m\atop \gcd(d, D)=1}(q^{m/d}+\epsilon(m/d))\mu(d).$$
From definition, $n(j)=\mathcal L(m)$ and so 
$$\n_A(Dm)=\sum_{j\le D-1\atop \gcd(j, D)=1}n(j)=\varphi(D)\cdot \mathcal L(m)=\frac{\varphi(D)}{Dm}\sum_{d|m\atop \gcd(d, D)=1}(q^{m/d}+\epsilon(m/d))\mu(d).$$
\end{proof}

\section{Conclusion}
In this paper, we have provided a short survey on a special action of $\PGL_2(\F_q)$ on irreducible polynomials over $\F_q$, establishing main results on the invariant theory of this group action: the characterization and number of fixed elements. We have noticed that a full characterization of the elements fixed by cyclic subgroups $\PGL_2(\F_q)$ was given in \cite{ST12}. We have presented some recent results, regarding enumeration formulas for the number of invariants and a natural correspondence between quadratic transformations and polynomials that are fixed by involutions $[A]\in \PGL_2(\F_q)$. Our contributions to the study of this action include explicit formulas for the number of fixed elements (in full generality), the characterization of the set of non-trivial invariants for any noncyclic subgroup of $\PGL_2(\F_q)$ and also a general correspondence between the polynomials fixed by an arbitrary element of $\PGL_2(\F_q)$ and rational transformations.

\begin{center}{\bf Acknowledgments}\end{center}
This work was conducted during a visit to Carleton University, supported by the Program CAPES-PDSE (process - 88881.134747/2016-01).

\end{document}